\documentclass[11pt, a4paper]{article}

\usepackage{geometry}
\usepackage{CJKutf8}
\usepackage{authblk}
\usepackage{listings}

\usepackage{amssymb}
\usepackage{amsmath}
\usepackage{amsthm}
\usepackage{amsfonts}
\usepackage{mathtools}
\usepackage{mathrsfs}
\usepackage{bm} 

\usepackage{graphicx,subfig}
\usepackage{dcolumn}
\usepackage[ruled,lined,algonl]{algorithm2e}
\usepackage{color}
\usepackage{autobreak}
\allowdisplaybreaks

\usepackage{tikz}
\usetikzlibrary{positioning, shapes.geometric}
\usepackage{setspace}

\theoremstyle{plain}
\newtheorem{theorem}{Theorem}

\newtheorem{lemma}{Lemma}
\newtheorem{proposition}{Proposition}

\theoremstyle{definition}
\newtheorem{definition}{Definition}

\theoremstyle{remark}
\newtheorem{remark}{Remark}

\DeclareMathOperator{\Det}{Det}
\DeclareMathOperator{\Tr}{Tr}
\DeclareMathOperator{\res}{res}

\newcommand{\pset}[1]{\mathcal{#1}}
\DeclareMathOperator{\discr}{discr}
\newcommand{\uvar}{\bm{u}}
\newcommand{\xvar}{\bm{x}}
\newcommand{\yvar}{\bm{y}}
\newcommand{\pvar}{\bm{p}}

\begin{document}
\begin{CJK}{UTF8}{gbsn}

\title{Stability and chaos of the duopoly model of Kopel: A study based on symbolic computations}


\author[a]{Xiaoliang Li}

\author[b]{Kongyan Chen}

\author[c,d]{Wei Niu}

\author[e]{Bo Huang\thanks{Corresponding author: huangbo0407@126.com}}

\affil[a]{School of Business, Guangzhou College of Technology and Business,\newline Guangzhou 510850, China}

\affil[b]{School of Digital Economics, Dongguan City University, Dongguan 523419, China}

\affil[c]{Ecole Centrale de P\'ekin, Beihang University, Beijing 100191, China}

\affil[d]{Beihang Hangzhou Innovation Institute Yuhang, Hangzhou 310051, China}

\affil[e]{LMIB -- School of Mathematical Sciences, Beihang University, Beijing 100191, China}

\date{}
\maketitle

\begin{abstract}
Since Kopel's duopoly model was proposed about three decades ago, there are almost no analytical results on the equilibria and their stability in the asymmetric case. The first objective of our study is to fill this gap. This paper analyzes the asymmetric duopoly model of Kopel analytically by using several tools based on symbolic computations. We discuss the possibility of the existence of multiple positive equilibria and establish necessary and sufficient conditions for a given number of positive equilibria to exist. The possible positions of the equilibria in Kopel's model are also explored. Furthermore, in the asymmetric model of Kopel, if the duopolists adopt the best response reactions or homogeneous adaptive expectations, we establish rigorous conditions for the local stability of equilibria for the first time. The occurrence of chaos in Kopel's model seems to be supported by observations through numerical simulations, which, however, is challenging to prove rigorously. The second objective is to prove the existence of snapback repellers in Kopel's map, which implies the existence of chaos in the sense of Li--Yorke according to Marotto's theorem.

\vspace{10pt}
\end{abstract}

\section{Introduction}

The study of oligopolistic competition lies at the core of the field of industrial organization.  Since decades ago, many economists have been making efforts to model and investigate oligopolistic competition by discrete dynamical systems of form with non-monotonic reaction functions. See, e.g., \cite{cavalli_nonlinear_2015, kopel_simple_1996, li_stability_2023, Puu1991C, tramontana_local_2015}. There are two strands of establishing non-monotonic reaction functions. The first one is to introduce nonlinearity to the demand function including \cite{cavalli_nonlinear_2015, Puu1991C, tramontana_local_2015}, etc. Whereas, the second one is to employ nonlinear cost functions. For example, Kopel \cite{kopel_simple_1996} proposed a famous Cournot duopoly model by assuming a linear demand function and nonlinear cost functions, which lead to unimodal reaction functions. In brief, Kopel's model can be described as the following map.
\begin{equation}\label{eq:map-kopel}
\left\{
\begin{split}
&x(t+1)=(1-\rho_1)x(t)+ \rho_1\mu_1 y(t)(1-y(t)),\\
&y(t+1)=(1-\rho_2)y(t)+ \rho_2\mu_2 x(t)(1-x(t)),
\end{split}\right.
\end{equation}
where $0< \rho_1,\rho_2\leq 1$ and $\mu_1,\mu_2>0$. In this map, $x(t)$ and $y(t)$ denote the output quantities of the two duopolists, respectively. Moreover, $\rho_1$ and $\rho_2$ represent the adjustment coefficients of the adaptive expectations of the two firms, while $\mu_1$ and $\mu_2$ measure the intensity of the positive externality the actions of the player exert on the payoff of its rival. 

The model of Kopel is mathematically interesting because it couples two logistic maps together. After Kopel's model was proposed, a large number of contributions have been made to intensively investigate, extend, and generalize the model. For example, Agiza \cite{agiza_analysis_1999} established rigorous conditions for the stability of the equilibria and investigated the bifurcations and chaos of Kopel's model in the special case of $\rho_1=\rho_2$ and $\mu_1=\mu_2$. Bischi et al.\ \cite{bischi_multistability_2000} proved the existence of the multistability and cycle attractors in map \eqref{eq:map-kopel}. Bischi and Kopel \cite{bischi_equilibrium_2001} used the method of critical curves to analyze the global bifurcations and illustrate the qualitative changes in the topological structure of the basins of Kopel's model. Anderson et al.\ \cite{anderson_basins_2005} considered the basins of attraction in the case of three nontrivial Nash equilibria in Kopel's map and found that a circle, lines, and rectangles play a key role in determining the basins. Govaerts and Khoshsiar \cite{Govaerts2008} considered the synchronization of two dynamics parameters of Kopel's model and discovered the dynamics by computing numerically the critical normal form coefficients of all codimension-one and codimension-two bifurcation points. Detailed bifurcation analyses were also conducted in, e.g., \cite{Li2022C, yao_codimension-one_2022}, which may help in understanding the occurrence and the structure of bifurcation cascades. Furthermore, Zhang and Gao \cite{zhang_equilibrium_2019} extended the model by assuming that each firm could forecast its rival’s output through a straightforward extrapolative foresight technology. Elsadany and Awad \cite{elsadany_dynamical_2016} modified Kopel's game by assuming one player uses a naive expectation whereas the other employs an adaptive technique and applied the feedback control method to control the chaotic behavior. Torcicollo \cite{torcicollo_dynamics_2013} generalized map \eqref{eq:map-kopel} by introducing the self-diffusion and cross-diffusion terms, and acquired general properties such as boundedness and uniqueness. Rionero and Torcicollo \cite{rionero_stability_2014} took into account the territory where the outputs are in the market and the time-depending firms’ strategies by generalizing Kopel's discrete model to a non-autonomous reaction-diffusion binary system of PDEs.

Agiza et al.\ \cite{agiza_multistability_1999} proposed a similar game with three competing firms, and provided a numerical study of the basins of attraction for several coexisting stable Nash equilibria, which reveals the occurrence of global bifurcations where the basins are transformed from connected sets into non-connected ones. Eskandari et al.\ \cite{eskandari_codimension_2021} reconsidered this triopoly discrete model and investigated possible codimension-one and codimension-two bifurcations. In \cite{li_bifurcation_2023}, the first author of this paper and his coworkers analyzed the bifurcation sets and the critical normal forms of different types of bifurcations to detect the complexity of dynamics in the Kopel triopoly game. Moreover, they conducted numerical simulations to investigate representative orbits, chaotic indicators, Lyapunov exponents, and bifurcation continuation.

In the rich literature regarding the model of Kopel, it is shocking that there are almost no analytical discussions on the equilibria and their stability in the asymmetric case of $\mu_1\neq \mu_2$. The first primary goal of our study is to fill this gap. The major obstacle to the analytical study of the asymmetric case of $\mu_1\neq \mu_2$ is that the closed-form equilibria can not be obtained explicitly. We employ several tools based on symbolic computations such as the triangular decomposition \cite{wang_elimination_2001} and the resultant \cite{Mishra1993A} to get around this obstacle. We explore the possibility of the existence of multiple positive equilibria, which has received a lot of attention from economists. We establish necessary and sufficient conditions for a given number of equilibria to exist. If the two duopolists adopt the best response reactions ($\rho_1=\rho_2=1$) or homogeneous adaptive expectations ($\rho_1=\rho_2$), we acquire rigorous conditions for the existence of distinct numbers of positive equilibria for the first time. 

Chaotic dynamics in map \eqref{eq:map-kopel} could be discovered through observations through numerical simulations. However, it is challenging to prove the existence of chaos rigorously. In this regard, Wu et al.\ \cite{wu_complex_2010} gave a computer-assisted verification for the existence of the chaotic dynamics in Kopel's model by using the topological horseshoe theory. They discussed the dynamics under the 104th iteration of Kopel's map and found that there exists a horseshoe in this attractor. C{\'a}novas and Mu{\~n}oz-Guillermo \cite{Canovas2018O} detected the existence of chaos if the firms in the model of Kopel are homogeneous. In the case of homogeneous players, they reduced the equations to a one-dimensional model, where the topological entropy with fixed accuracy can be computed. It should be pointed out that the computations of the topological entropy both in \cite{wu_complex_2010} and \cite{Canovas2018O} are based on float-points algorithms, where the obtained results are not necessarily reliable in some sense. In this paper, by using tools based on symbolic computations, we rigorously prove the existence of snapback repellers in Kopel's map, which implies the existence of chaos in the sense of Li--Yorke \cite{li_period_1975} according to Marotto’s theorem \cite{marotto_snap-back_1978, marotto_redefining_2005}. We should mention that the main idea for symbolic analysis of snapback repellers was initially proposed in \cite{huang_analysis_2019} by Huang and Niu, coauthors of this paper. However, their method is insufficient and is fixed in this paper by introducing more steps of further validation for the existence of snapback repellers.

The rest of this paper is structured as follows. In Section 2, we explore the number of the equilibria in map \eqref{eq:map-kopel} and their possible positions. Section 3 analytically investigates the local stability of the equilibria in the asymmetric model of Kopel. In Section 4, we rigorously prove the existence of snapback repellers, hence chaos, in the symmetric case. Concluding remarks are provided in Section 5.

\section{Equilibrium Analysis}


To the best of our knowledge, Li et al.\ \cite{Li2022C} preliminarily explored the number of equilibria in the asymmetric case of $\mu_1\neq \mu_2$. In this section, we aim at providing a more systematic analysis of this case. By setting $x(t+1)=x(t)=x^*$ and $y(t+1)=y(t)=y^*$ in map \eqref{eq:map-kopel}, we obtain the equilibrium equations
\begin{equation}\label{eq:eq-equi}
\left\{
	\begin{split}
	& x^*=\mu_1 y^*(1-y^*),\\
	& y^*=\mu_2 x^*(1-x^*).	
	\end{split}
\right.
\end{equation}


Provided that $\mu_1,\mu_2\leq 4$, it is evident that $[0,1]\times [0,1]$ is a trapping set of map \eqref{eq:map-kopel}, which means that all the trajectories will stay in $[0,1]\times [0,1]$ if the initial belief $(x(0),y(0))$ is taken from $[0,1]\times [0,1]$. In addition, from the following Proposition \ref{prop:all-equi}, one can see that if $\mu_1\mu_2\geq 1$ (even if $\mu_1,\mu_2>4$), then the equilibria that the trajectories approach should lie in $[0,1]\times [0,1]$. In the proof of Proposition \ref{prop:all-equi}, the triangular decomposition method plays an ambitious role, which extends the Gaussian elimination method such that polynomial equations can be handled. Readers may refer to \cite{jin_new_2013, Li2010D, wang_elimination_2001} for more information regarding the triangular decomposition. 

Moreover, the notion of resultant is required in our study. Let $$F=\sum_{i=0}^ma_i\,x^i,\quad G=\sum_{j=0}^lb_j\,x^j$$ be two univariate polynomials in $x$ with coefficients $a_i,b_j$ in the field of complex numbers, and $a_m,b_l\neq 0$. The
determinant
\begin{equation*}\label{eq:sylmat}
 \begin{array}{c@{\hspace{-5pt}}l}
 \left|\begin{array}{cccccc}
a_m & a_{m-1}& \cdots   & a_0   &        &       \\
           & \ddots   & \ddots&    \ddots    &\ddots&   \\
         &          & a_m   & a_{m-1}&\cdots& a_0 \\ [5pt]
b_l & b_{l-1}& \cdots   &  b_0 &    &         \\
            & \ddots   &\ddots &   \ddots     &\ddots&       \\
         &   &    b_{l}     & b_{l-1} & \cdots &  b_0
\end{array}\right|
& \begin{array}{l}\left.\rule{0mm}{8mm}\right\}l\\
\\\left.\rule{0mm}{8mm}\right\}m
\end{array}
\end{array}
\end{equation*}
is called the \emph{Sylvester resultant} (or simply
\emph{resultant}) of $F$ and $G$, and denoted by $\res(F,G,x)$. The resultant of $F$ and its derivative ${\rm d}F/{\rm d}x$, i.e., $\res(F,{\rm d}F/{\rm d}x,x)$, is called the \emph{discriminant} of $F$ and denoted by $\discr(F)$. The following lemma is one of the well-known properties of resultants, which can also be found in \cite{Mishra1993A}.

%
%
%

\begin{lemma}\label{lem:res-com}
Let $A$ and $B$ be two univariate polynomials in $x$.
There exist two polynomials $F$ and $G$ in $x$ such that 
$$FA+GB=\res(A,B,x).$$
Furthermore, $A$ and $B$ have common zeros in the field of complex numbers if and only if $\res(A, B, x)=0$.
\end{lemma}


\begin{proposition}\label{prop:all-equi}
	In map \eqref{eq:map-kopel}, all the equilibria remain inside $[0,1]\times [0,1]$ if $\mu_1\mu_2\geq 1$.
\end{proposition}

\begin{proof}
By using the triangular decomposition method, we can decompose the solutions of system \eqref{eq:eq-equi} into zeros of two triangular sets, i.e., $[y^*,x^*]$ and $[G_2,G_1]$,
where
\begin{align*}
	&G_1=\mu_1 \mu_2^2 x^{*3}- 2\, \mu_1 \mu_2^{2} x^{*2}+\left(\mu_1 \mu_2^{2}+\mu_1 \mu_2\right) x^*-\mu_1 \mu_2 +1,\\
	&G_2=y^*+\mu_2 x^{*2} -\mu_2 x^*.
\end{align*}
The zero of $[y^*,x^*]$ corresponds to the equilibrium $E_0=(0,0)$, which obviously lies in $[0,1]\times [0,1]$. Therefore, we focus on the other branch $[G_2, G_1]$. 

Because of the symmetry of $x^*,y^*$ and $\mu_1,\mu_2$ in \eqref{eq:eq-equi}, we just need to prove that any zero of $G_1$ satisfies $x^*\in [0,1]$ if $\mu_1\mu_2\geq 1$. One can compute that $\res(G_1,1-x^*,x^*)=1$. According to Lemma \ref{lem:res-com}, we know any zero $x^*$ of $G_1$ can not touch the line $x^*=1$ as $\mu_1$ and $\mu_2$ vary. Therefore, any zero $x^*$ of $G_1$ is smaller than 1. Furthermore, we have $\res(G_1,x^*,x^*)=\mu_1\mu_2-1$. This means that the sign of a zero of $G_1$ may change as the parameter point $(\mu_1,\mu_2)$ goes across the curve $\mu_1\mu_2=1$. It is easy to check that any zero of $G_1$ is positive if $\mu_1\mu_2>1$ and vice versa. When $\mu_1\mu_2=1$, we can plug $\mu_2=1/\mu_1$ into $G_1$ and solve three complex zeros $0$, $1+\sqrt{-\mu_1}$, and $1-\sqrt{-\mu_1}$. The only real zero $x^*=0$ is in $[0,1]$, obviously. In conclusion, all zeros of $G_1$ remain inside $[0,1]$ if $\mu_1\mu_2\geq 1$, which completes the proof.
\end{proof}

From an economic point of view, positive equilibria that satisfy $x^*,y^*>0$ are more important and mainly focused on in what follows.

\begin{theorem}\label{thm:equi-num}
Denote 
$$R_1=\mu_1^2 \mu_2^2-4\,\mu_1^2 \mu_2-4\,\mu_1 \mu_2^2 +18\,\mu_1\mu_2 -27.$$ 
Three distinct positive equilibria exist if and only if $R_1>0$, and one unique positive equilibrium exists if and only if $R_1<0$ and $\mu_1\mu_2>1$. In the case of $R_1=0$, two distinct positive equilibria exist if and only if $(\mu_1,\mu_2)\neq (3,3)$, and one unique positive equilibrium $\left(\frac{2}{3},\frac{2}{3}\right)$ exists if $(\mu_1,\mu_2)= (3,3)$.
\end{theorem}

\begin{proof}
One can see that $G_2$ is linear with respect to $y^*$, which means that the number of equilibria corresponding to $[G_2, G_1]$ equals that of distinct real zeros of $G_1$. Consequently, we analyze the real zeros of $G_1$ in the sequel. 

For a univariate polynomial $F(x)$, the multiplicity of a zero $\bar x$ is greater than 1 if $F(\bar x)=0$ and $F'(\bar x)=0$ are simultaneously satisfied, where $F'$ is the derivative of $F$. We have $\res(G_1,G'_1,x^*)=-\mu_1^{3} \mu_2^{6} R_1$.
Therefore, the nature (e.g., the multiplicity and number) of real zeros of $G_1$ should not change if the parameter point $(\mu_1,\mu_2)$ is not across the curve $R_1=0$ as the values of $\mu_1,\mu_2$ vary. As shown in Figure \ref{fig:num-equilibria}, the curve $R_1=0$ (the red curve) divides the parameter set of concern $\{(\mu_1,\mu_2)\,|\,\mu_1,\mu_2>0\}$ into two regions. In the dark-gray region defined by $R_1>0$, the nature of the zeros can be determined by checking at some sample point, say $(4,4)$ for example. At $(4,4)$, we have
$$G_1=64\, x^{*3}-128\, x^{*2}+80\, x^*-15,$$ 
which can be solved by $\frac{3}{4},\frac{5}{8}-\frac{\sqrt{5}}{8},\frac{5}{8}+\frac{\sqrt{5}}{8}$. Hence, three distinct real zeros of $G_1$ exist if $R_1>0$. This implies that three distinct positive equilibria exist in map \eqref{eq:map-kopel} if $R_1>0$.

In the other region defined by $R_1<0$, we similarly select a sample point, e.g., $(2,2)$. At this point, we have 
$$G_1=8\, x^{*3}-16\, x^{*2}+12\, x^*-3,$$ 
which can be solved by $\frac{1}{2},\frac{3}{4}-\frac{\mathrm{I} \sqrt{3}}{4},\frac{3}{4}+\frac{\mathrm{I} \sqrt{3}}{4}$. Thus, only one real solution of $G_1$ exists if $R_1<0$. However, in the proof of Proposition \ref{prop:all-equi}, it has been proved that the zero of $G_1$ is positive if $\mu_1\mu_2>1$. Therefore, one unique positive equilibrium exists if $R_1<0$ and $\mu_1\mu_2>1$.

Furthermore, on the curve $R_1=0$, we have $G_1=G'_1=0$ according to Lemma \ref{lem:res-com}. Then, the multiplicity of some zeros of $G_1$ should be greater than 1, which means that some zeros of $G_1$ become identical as the parameter point $(\mu_1,\mu_2)$ varies from the region $R_1>0$ to its edge $R_1=0$. The degree of $G_1$ with respect to $x^*$ is 3, thus the multiplicity of a zero of $G_1$ can be taken as 3 at maximum. In this case, $G_1=0$, $G_1'=0$, and $G''_1=0$ should be fulfilled simultaneously. One can compute 
$$\res(G_1,G_1'',x^*)=-8\, \mu_1^{3} \mu_2^{6} \left(2\, \mu_1 \mu_2^{2}-9\, \mu_1 \mu_2+27\right).$$ 
By solving $2\, \mu_1 \mu_2^{2}-9\, \mu_1 \mu_2+27=0$ and $R_1=0$, we obtain $(\mu_1,\mu_2)=(3,3)$, where $G_1$ has one unique zero  $x^*=2/3$ with multiplicity 3. Thus, one unique positive equilibrium $\left(\frac{2}{3},\frac{2}{3}\right)$ exists if $(\mu_1,\mu_2)= (3,3)$. But, if $(\mu_1,\mu_2)\neq (3,3)$ and $R_1=0$, then $G_1=0$, $G_1'=0$, and $G''_1\neq 0$. In this case, there are three real zeros of $G_1$, among which two are equal. This means that there are two distinct positive equilibria. The proof is complete.
\end{proof}

We summarize the above results in Figure \ref{fig:num-equilibria}. The regions where three and one unique positive equilibria exist are marked in dark-gray and light-gray, respectively.

\begin{figure}[htbp]
    \centering
    \includegraphics[width=0.45\textwidth]{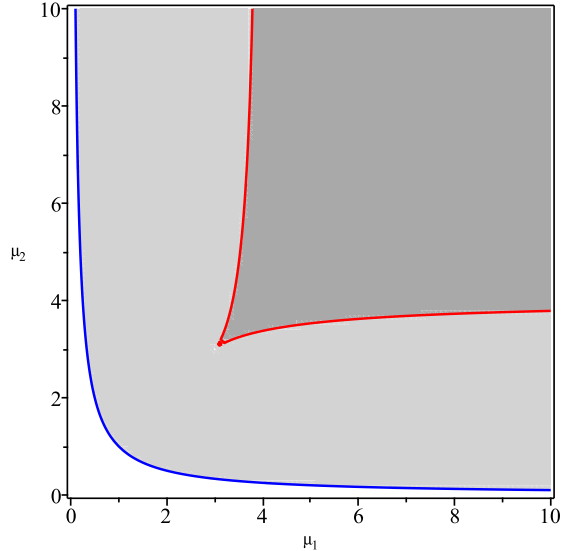}
    \caption{Partitions of the parameter set $\{(\mu_1,\mu_2)\,|\,\mu_1,\mu_2>0\}$ of map \eqref{eq:map-kopel} for distinct numbers of equilibria. The blue and red curves are defined by $\mu_1\mu_2=1$ and $R_1=0$, respectively. In the dark-gray and light-gray regions, there are three and one positive equilibria, respectively.}
    \label{fig:num-equilibria}
\end{figure}

\section{Local Stability}

The Jacobian matrix of map \eqref{eq:map-kopel} at an equilibrium $(x^*,y^*)$ is 

\begin{equation*}
J(x^*,y^*)=\left[
\begin{matrix}
1-\rho_1 &  \mu_1 \rho_1 (1-2\,y^*)
\\
 \mu_2 \rho_2 (1-2\,x^*) & 1-\rho_2
\end{matrix}
\right].
\end{equation*}
Then, the characteristic polynomial of $J$ becomes
$$CP(\lambda)=\lambda^2-\Tr(J)\lambda+\Det(J),$$
where
\begin{align*}
	&\Tr(J)=2-\rho_1-\rho_2,\\
	&\Det(J)=(1-\rho_1)(1-\rho_2)-\mu_1\mu_2\rho_1\rho_2(1-2\,x^*)(1-2\,y^*),
\end{align*}
are the trace and the determinant of $J$, respectively.
According to the Jury criterion \cite{Jury1976I}, the conditions for the local stability include:
\begin{enumerate}
	\item $CD_1^J\equiv CP(1)= 1-\Tr(J)+\Det(J)>0$,
	\item $CD_2^J\equiv CP(-1)= 1+\Tr(J)+\Det(J)>0$,
	\item $CD_3^J\equiv 1-\Det(J)>0$.
\end{enumerate}

For the zero equilibrium $E_0=(0,0)$, we have
\begin{align*}
	&CD_1^J(E_0)=\rho_1 \rho_2 (1-\mu_1\mu_2),\\
	&CD_2^J(E_0)= \rho_1 \rho_2( 1-\mu_1\mu_2)-2\, \rho_1-2\,\rho_2+4,\\
	&CD_3^J(E_0)= \rho_1 \rho_2 (\mu_1\mu_2-1)+\rho_1+\rho_2.
\end{align*}
It can be easily derived that $E_0$ is locally stable if 
$$1-\frac{\rho_1+ \rho_2}{\rho_1 \rho_2}< \mu_1\mu_2 <1.$$ 
In what follows, we discuss the positive equilibria. 

It is known that the discrete dynamic system may undergo a local bifurcation when the equilibrium loses its stability at $CD_1^J=0$, $CD_2^J=0$, or $CD_3^J=0$. Denote $\pset{G}=[G_2,G_1]$. Firstly, we consider the possibility of $CD_1^J=0$ by computing
$$\res(CD_1^J,\pset{G})\equiv \res(\res(CD_1^J,G_2,y^*),G_1,x^*).$$
It is obtained that 
$$\res(CD_1^J,\pset{G})=-\rho_1^{3} \rho_2^{3} \mu_1^{3} \mu_2^{6} \left(\mu_1\mu_2-1\right) R_1.$$
According to Lemma \ref{lem:res-com}, $CD_1^J=0$ and $\pset{G}=0$ imply $\res(CD_1^J,\pset{G})=0$. Hence, at a positive equilibrium of map \eqref{eq:map-kopel}, a necessary condition for $CD_1^J=0$ is $\mu_1\mu_2=1$ or $R_1=0$. Similarly, regarding the possibility of $CD_2^J=0$ and $CD_3^J=0$, we have
\begin{align*}
	&\res(CD_2^J,\pset{G})=-\mu_1^3 \mu_2^6 R_2,\\
	&\res(CD_3^J,\pset{G})=\mu_1^3 \mu_2^6 R_3,
\end{align*}
where
\begin{align*}
	\begin{autobreak}
R_2=
\rho_1^{3} \rho_2^{3} \mu_1^{3} \mu_2^{3} 
- 4\,\rho_1^{3} \rho_2^{3} \mu_1^{3} \mu_2^{2} 
- 4\,\rho_1^{3} \rho_2^{3} \mu_1^{2} \mu_2^{3} 
+ 17\,\rho_1^{3} \rho_2^{3} \mu_1^{2} \mu_2^{2} 
+ 4\,\rho_1^{3} \rho_2^{3} \mu_1^{2} \mu_2 
+ 4\,\rho_1^{3} \rho_2^{3} \mu_1 \mu_2^{2} 
- 2\,\rho_1^{3} \rho_2^{2} \mu_1^{2} \mu_2^{2} 
- 2\,\rho_1^{2} \rho_2^{3} \mu_1^{2} \mu_2^{2} 
- 45\,\rho_1^{3} \rho_2^{3} \mu_1 \mu_2 
+ 8\,\rho_1^{3} \rho_2^{2} \mu_1^{2} \mu_2 
+ 8\,\rho_1^{3} \rho_2^{2} \mu_1 \mu_2^{2} 
+ 8\,\rho_1^{2} \rho_2^{3} \mu_1^{2} \mu_2 
+ 8\,\rho_1^{2} \rho_2^{3} \mu_1 \mu_2^{2} 
+ 4\,\rho_1^{2} \rho_2^{2} \mu_1^{2} \mu_2^{2} 
- 36\,\rho_1^{3} \rho_2^{2} \mu_1 \mu_2 
- 36\,\rho_1^{2} \rho_2^{3} \mu_1 \mu_2 
- 16\,\rho_1^{2} \rho_2^{2} \mu_1^{2} \mu_2 
- 16\,\rho_1^{2} \rho_2^{2} \mu_1 \mu_2^{2} 
+ 27\,\rho_1^{3} \rho_2^{3} 
- 4\,\rho_1^{3} \rho_2 \mu_1 \mu_2 
+ 64\,\rho_1^{2} \rho_2^{2} \mu_1 \mu_2 
- 4\,\rho_1 \rho_2^{3} \mu_1 \mu_2 
+ 54\,\rho_1^{3} \rho_2^{2} 
+ 54\,\rho_1^{2} \rho_2^{3} 
+ 16\,\rho_1^{2} \rho_2 \mu_1 \mu_2 
+ 16\,\rho_1 \rho_2^{2} \mu_1 \mu_2 
+ 36\,\rho_1^{3} \rho_2 
- 36\,\rho_1^{2} \rho_2^{2} 
+ 36\,\rho_1 \rho_2^{3} 
- 16\,\rho_1 \rho_2 \mu_1 \mu_2 
+ 8\,\rho_1^{3} 
- 120\,\rho_1^{2} \rho_2 
- 120\,\rho_1 \rho_2^{2} 
+ 8\,\rho_2^{3} 
- 48\,\rho_1^{2} 
+ 48\,\rho_1 \rho_2 
- 48\,\rho_2^{2} 
+ 96\,\rho_1 
+ 96\,\rho_2 
- 64,
\end{autobreak}\\

\begin{autobreak}
R_3=
\rho_1^{3} \rho_2^{3} \mu_1^{3} \mu_2^{3} 
- 4\,\rho_1^{3} \rho_2^{3} \mu_1^{3} \mu_2^{2} 
- 4\,\rho_1^{3} \rho_2^{3} \mu_1^{2} \mu_2^{3} 
+ 17\,\rho_1^{3} \rho_2^{3} \mu_1^{2} \mu_2^{2} 
+ 4\,\rho_1^{3} \rho_2^{3} \mu_1^{2} \mu_2 
+ 4\,\rho_1^{3} \rho_2^{3} \mu_1 \mu_2^{2} 
- \rho_1^{3} \rho_2^{2} \mu_1^{2} \mu_2^{2} 
- \rho_1^{2} \rho_2^{3} \mu_1^{2} \mu_2^{2} 
- 45\,\rho_1^{3} \rho_2^{3} \mu_1 \mu_2 
+ 4\,\rho_1^{3} \rho_2^{2} \mu_1^{2} \mu_2 
+ 4\,\rho_1^{3} \rho_2^{2} \mu_1 \mu_2^{2} 
+ 4\,\rho_1^{2} \rho_2^{3} \mu_1^{2} \mu_2 
+ 4\,\rho_1^{2} \rho_2^{3} \mu_1 \mu_2^{2} 
- 18\,\rho_1^{3} \rho_2^{2} \mu_1 \mu_2 
- 18\,\rho_1^{2} \rho_2^{3} \mu_1 \mu_2 
+ 27\,\rho_1^{3} \rho_2^{3} 
- \rho_1^{3} \rho_2 \mu_1 \mu_2 
- 2\,\rho_1^{2} \rho_2^{2} \mu_1 \mu_2 
- \rho_1 \rho_2^{3} \mu_1 \mu_2 
+ 27\,\rho_1^{3} \rho_2^{2} 
+ 27\,\rho_1^{2} \rho_2^{3} 
+ 9\,\rho_1^{3} \rho_2 
+ 18\,\rho_1^{2} \rho_2^{2} 
+ 9\,\rho_1 \rho_2^{3} 
+ \rho_1^{3} 
+ 3\,\rho_1^{2} \rho_2 
+ 3\,\rho_1 \rho_2^{2} 
+ \rho_2^{3}.
\end{autobreak}
\end{align*}

Therefore, local bifurcations of the equilibria may take place when $\mu_1\mu_2=1$ or $R_i=0$, $i=1,2,3$. However, we should mention that the signs of $CD_i^J$ and those of $\res(CD_i^J,\pset{G})$ may not be the same for a given $i$. Even if the signs of $\res(CD_i^J,\pset{G})$ are fixed, the signs of $CD_i^J$ may change. For example, if keeping $\rho_1=\rho_2=1$, we have $\res(R_1,\pset{G})<0$, $\res(R_2,\pset{G})<0$ and $\res(R_3,\pset{G})>0$ when $\mu_1=\mu_2=4$ or $\mu_1=\mu_2=\frac{13}{4}$. When $\mu_1=\mu_2=4$, there are three equilibria, i.e., $\left(\frac{3}{4},\frac{3}{4}\right)$, $\left(\frac{5}{8}-\frac{\sqrt{5}}{8},\frac{5}{8}+\frac{\sqrt{5}}{8}\right)$, and $\left(\frac{5}{8}+\frac{\sqrt{5}}{8},\frac{5}{8}-\frac{\sqrt{5}}{8}\right)$, where the signs of $CD_3^J$ are $+$, $-$, and $-$, respectively. Whereas, when $\mu_1=\mu_2=\frac{13}{4}$, there are three equilibria $\left(\frac{9}{13},\frac{9}{13}\right)$, $\left(\frac{17}{26}-\frac{\sqrt{17}}{26},\frac{17}{26}+\frac{\sqrt{17}}{26}\right)$, and $\left(\frac{17}{26}+\frac{\sqrt{17}}{26},\frac{17}{26}-\frac{\sqrt{17}}{26}\right)$, where the signs of $CD_3^J$ are $+$, $+$, and $+$, respectively.

\begin{remark}\label{rem:idea}
However, it can be derived that if we vary the parameter point continuously without going across the algebraic variety $\res(CD_i^J,\pset{G})=0$ in the parameter space, then $CD_i^J$ will not annihilate (become zero) and its sign will keep fixed. This means that in each region divided by $\res(CD_i^J,\pset{G})=0$, $i=1,2,3$, the signs of $CD_i$, $i=1,2,3$, will not change and we can determine the number of real zeros satisfying $CD_i>0$, $i=1,2,3$, by selecting at least one sample point. 	
\end{remark}

Firstly, we consider the case of $\rho_1=\rho_2=1$. This case is much simpler but is of tremendous economic interest itself, in which the best response reaction functions are actually adopted by the two duopolists. We denote 
\begin{align*}
	&S_2\equiv R_2|_{\rho_1=\rho_2=1}=\left(\mu_1\mu_2-1\right) R_1,\\
	&S_3\equiv R_3|_{\rho_1=\rho_2=1}=\mu_1^{3} \mu_2^{3}-4\, \mu_1^{3} \mu_2^{2}-4\, \mu_1^{2} \mu_2^{3}+15\, \mu_1^{2} \mu_2^{2}+12\, \mu_1^{2} \mu_2+12\, \mu_1 \mu_2^{2}-85\, \mu_1 \mu_2+125.
\end{align*}

From Theorem \ref{thm:equi-num}, we recall that the number of positive equilibria will remain the same if we vary the parameter point continuously without being across the curves $\mu_1\mu_2=1$ and $R_1=0$ in the parameter space. Furthermore, according to Remark \ref{rem:idea}, it is known that the number of stable ($CD_i^J>0$) positive equilibria will not change if the parameter point does not go across $\mu_1\mu_2=1$, $R_1=0$,  $S_2=0$, $S_3=0$ (or simply $\mu_1\mu_2=1$, $R_1=0$, $S_3=0$ as $S_2=(\mu_1\mu_2-1)R_1$). Hence, one can select at least one sample point from each region divided by $\mu_1\mu_2=1$, $R_1=0$, $S_3=0$, and then identify the number of stable positive equilibria at the sample point. The process of selecting sample points can be automated by using, e.g., the method of partial cylindrical algebraic decomposition or called the PCAD method \cite{Collins1991P}. Herein, the PCAD method permits us to select 39 sample points such that there is at least (not exactly) one sample point from each of these regions. We list the selected sample points in Table \ref{tab:sample}, where ``num" stands for the number of stable positive equilibria. Moreover, Table \ref{tab:sample} reports the signs of $\mu_1\mu_2-1$, $R_1$, and $S_3$ at all these sample points. 

\begin{table}[htbp]
	\centering 
	\caption{Selected Sample Points in $\{(\mu_1,\mu_2)\,|\,\mu_1,\mu_2>0\}$}
	\label{tab:sample} 
	\begin{tabular}{|l|c|c|c|c||l|c|c|c|c|}  
\hline  
$(\mu_1,\mu_2)$ & num & $\mu_1\mu_2-1$ & $R_1$ & $S_3$ & $(\mu_1,\mu_2)$ & num & $\mu_1\mu_2-1$ & $R_1$ & $S_3$ \\ \hline
$(1, 1/2)$ & 0 & $-$ & $-$ & $+$ & $(1, 2)$ & 1 & $+$ & $-$ & $+$ \\ \hline
$(1, 5)$ & 0 & $+$ & $-$ & $-$ & $(25/8, 1/4)$ & 0 & $-$ & $-$ & $+$ \\ \hline
$(25/8, 1)$ & 1 & $+$ & $-$ & $+$ & $(25/8, 25/8)$ & 2 & $+$ & $+$ & $+$ \\ \hline
$(25/8, 13/4)$ & 1 & $+$ & $-$ & $+$ & $(25/8, 4)$ & 0 & $+$ & $-$ & $-$ \\ \hline
$(13/4, 1/4)$ & 0 & $-$ & $-$ & $+$ & $(13/4, 1)$ & 1 & $+$ & $-$ & $+$ \\ \hline
$(13/4, 13/4)$ & 2 & $+$ & $+$ & $+$ & $(13/4, 55/16)$ & 1 & $+$ & $+$ & $-$ \\ \hline
$(13/4, 4)$ & 0 & $+$ & $-$ & $-$ & $(27/8, 1/4)$ & 0 & $-$ & $-$ & $+$ \\ \hline
$(27/8, 1)$ & 1 & $+$ & $-$ & $+$ & $(27/8, 2)$ & 0 & $+$ & $-$ & $-$ \\ \hline
$(27/8, 3)$ & 1 & $+$ & $-$ & $+$ & $(27/8, 13/4)$ & 2 & $+$ & $+$ & $+$ \\ \hline
$(27/8, 7/2)$ & 1 & $+$ & $+$ & $-$ & $(27/8, 5)$ & 0 & $+$ & $-$ & $-$ \\ \hline
$(219/64, 1/4)$ & 0 & $-$ & $-$ & $+$ & $(219/64, 1)$ & 1 & $+$ & $-$ & $+$ \\ \hline
$(219/64, 2)$ & 0 & $+$ & $-$ & $-$ & $(219/64, 103/32)$ & 1 & $+$ & $+$ & $-$ \\ \hline
$(219/64, 13/4)$ & 2 & $+$ & $+$ & $+$ & $(219/64, 7/2)$ & 1 & $+$ & $+$ & $-$ \\ \hline
$(219/64, 5)$ & 0 & $+$ & $-$ & $-$ & $(7/2, 1/4)$ & 0 & $-$ & $-$ & $+$ \\ \hline
$(7/2, 1)$ & 1 & $+$ & $-$ & $+$ & $(7/2, 2)$ & 0 & $+$ & $-$ & $-$ \\ \hline
$(7/2, 13/4)$ & 1 & $+$ & $+$ & $-$ & $(7/2, 7/2)$ & 0 & $+$ & $+$ & $+$ \\ \hline
$(7/2, 4)$ & 1 & $+$ & $+$ & $-$ & $(7/2, 5)$ & 0 & $+$ & $-$ & $-$ \\ \hline
$(5, 1/8)$ & 0 & $-$ & $-$ & $+$ & $(5, 1/2)$ & 1 & $+$ & $-$ & $+$ \\ \hline
$(5, 1)$ & 0 & $+$ & $-$ & $-$ & $(5, 57/16)$ & 1 & $+$ & $+$ & $-$ \\ \hline
$(5, 4)$ & 0 & $+$ & $+$ & $+$ & &  &  &  &  \\ \hline
	\end{tabular}
\end{table}

From Table \ref{tab:sample}, we can see that one unique stable positive equilibrium exists if $\mu_1\mu_2>1$, $R_1<0$, $S_3>0$ or $\mu_1\mu_2>1$, $R_1>0$, $S_3<0$. However, the conditions for the existence of two stable positive equilibria are not evident. For example, two stable positive equilibria exist at $(\mu_1,\mu_2)=(13/4,13/4)$, where $\mu_1\mu_2>1$, $R_1>0$, and $S_3>0$ are satisfied; whereas no stable positive equilibria exist at $(\mu_1,\mu_2)=(5,4)$, where $\mu_1\mu_2>1$, $R_1>0$, and $S_3>0$ are also fulfilled. 
 
We acquire the regions for distinct numbers of stable positive equilibria, which are depicted in Figure \ref{fig:num-stable-eq}. The region of two stable positive equilibria is marked in yellow, while that of one stable positive equilibrium is marked in light-gray. Indeed, it may be quite complicated to algebraically describe a region bounded by algebraic varieties. For example, the yellow region can not be characterized by inequalities regarding $\mu_1\mu_2-1$, $R_1$, and $S_3$ (their signs at the yellow region are the same as those at the white region where $(4,4)$ lies). However, this yellow region can be described algebraically by introducing additional polynomials, which can be found in the so-called generalized discriminant list. Readers may refer to \cite{Li2014C, Yang2001A} for more information.

We formally summarize the obtained results in Theorem \ref{thm:stable-equiv}, where $A_1$ and $A_2$ are additional polynomials picked out from the generalized discriminant list. It should be mentioned the computations of searching for these polynomials are quite expensive although this process works perfectly in theory. 

\begin{figure}[htbp]
  \centering
  \subfloat[$(\mu_1,\mu_2)\in (0,10)\times(1,10)$]{\includegraphics[width=0.45\textwidth]{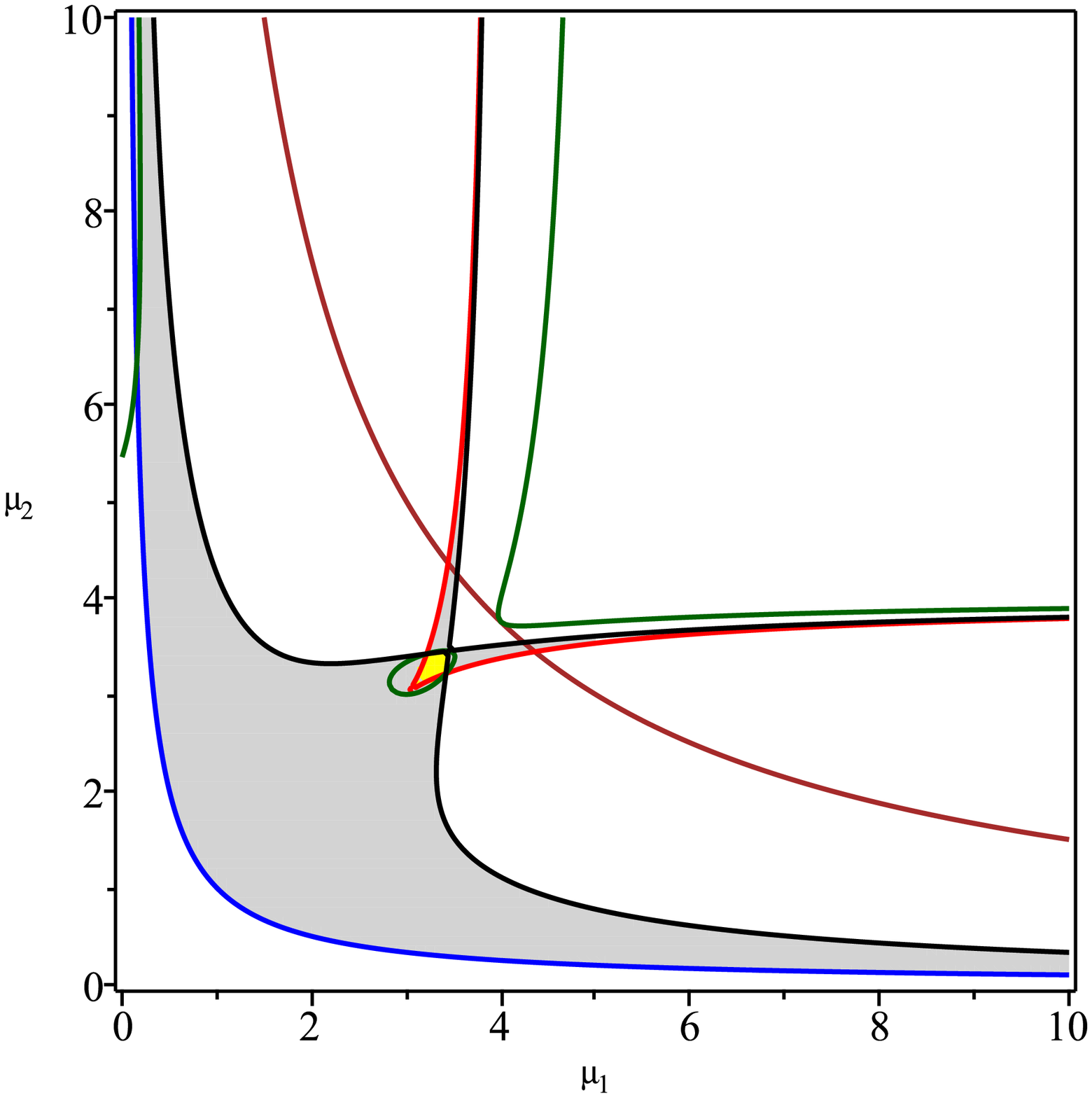}} 
  \subfloat[$(\mu_1,\mu_2)\in (2.5,5)\times(2.5,5)$]{\includegraphics[width=0.45\textwidth]{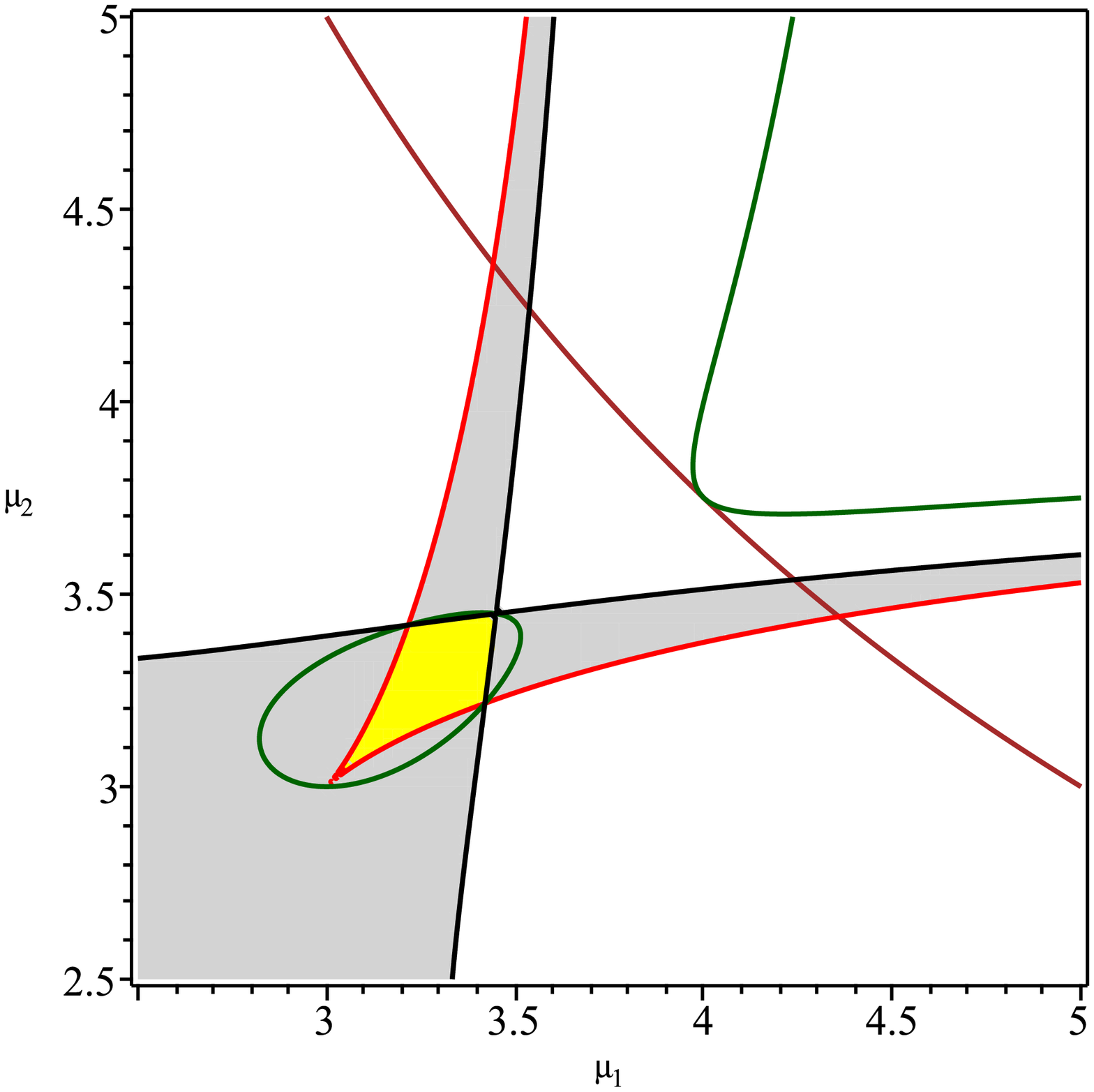}} \\

  \caption{Partitions of the parameter set $\{(\mu_1,\mu_2)\,|\,\mu_1,\mu_2>0\}$ in the case of $\rho_1=\rho_2=1$ for distinct numbers of stable positive equilibria. The blue, red, black, brown, and green curves are defined by $\mu_1\mu_2=1$, $R_1=0$, $S_3=0$, $A_1=0$, and $A_2=0$, respectively. In the yellow and light-gray regions, there are two and one stable positive equilibria, respectively.}
\label{fig:num-stable-eq}
\end{figure}

%

\begin{theorem}\label{thm:stable-equiv}
	In the case of $\rho_1=\rho_2=1$, two stable positive equilibria exist if $R_1>0$, $S_3>0$, $A_1<0$, and $A_2>0$, where
\begin{align*}
&A_1=\mu_1\mu_2-15,\\
&A_2=\mu_1^2\mu_2^2-4\,\mu_1^2 \mu_2-5\, \mu_1\mu_2^2+21\,\mu_1\mu_2+11\,\mu_2-60.
\end{align*}
Furthermore, one unique stable positive equilibrium exists if $\mu_1\mu_2>1$, $R_1<0$, $S_3>0$ or $\mu_1\mu_2>1$, $R_1>0$, $S_3<0$. 
\end{theorem}

In the more general case of $\rho_1=\rho_2\in(0,1]$, homogeneous adaptive expectations are employed by the two duopolists. Similar calculations can be conducted, where 4947 sample points are selected. Details are not reported herein due to space limitations. Readers can understand that the algebraic description of the region where two stable positive equilibria exist can not be computed in a reasonably short time according to our implementations of the methods. However, we can geometrically describe the region of two stable positive equilibria: it is bounded by $R_1=0$ (the red surface), $H_3=0$ (the blue surface), $\rho_1=0$, and $\rho_1=1$ (See Figure \ref{fig:3d-2stable}). For the region where one stable positive equilibrium exists, our computations can yield the results, which are reported in Theorem \ref{thm:ab-one}.

\begin{figure}[htbp]
    \centering
    \includegraphics[width=10cm]{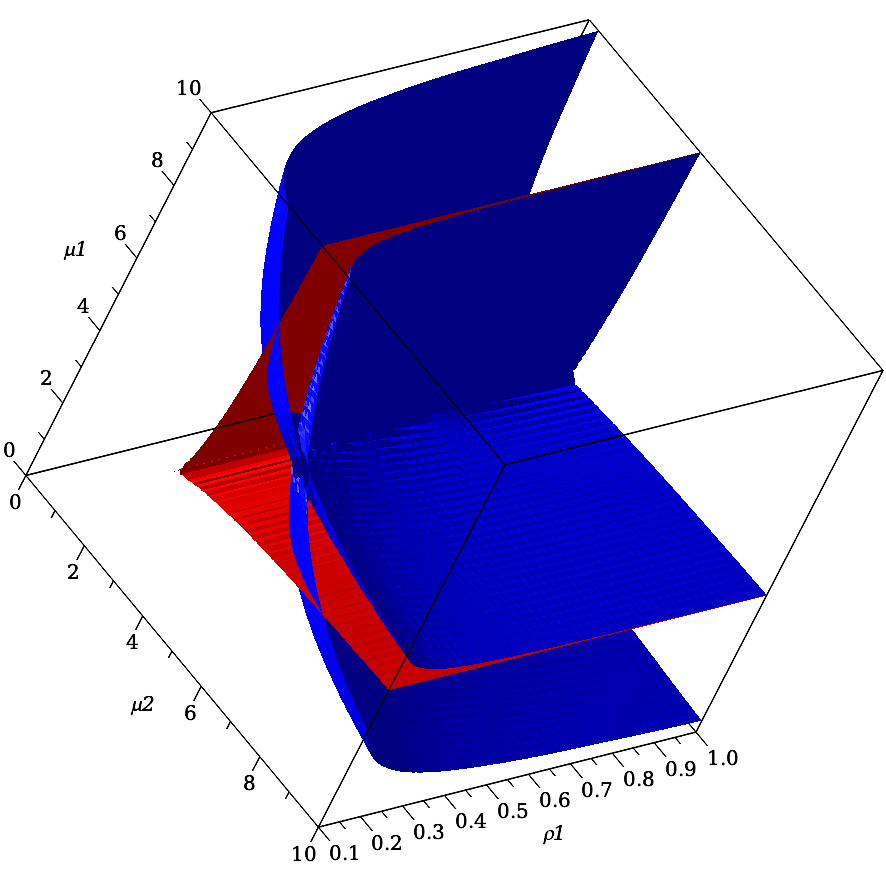}
    \caption{The parametric set $\{(\mu_1,\mu_2,\rho_1)\,|\,\mu_1,\mu_2>0,0<\rho_1\leq 1\}$ of map \eqref{eq:map-kopel} in the case of $\rho_1=\rho_2$. The surfaces of $R_1=0$ and $H_3=0$ are colored in red and blue, respectively. The region where two stable positive equilibria exist is bounded by the surfaces $R_1=0$, $H_3=0$, $\rho_1=0$, and $\rho_1=1$.}
    \label{fig:3d-2stable}
\end{figure}

\begin{theorem}\label{thm:ab-one}
In the case of $\rho_1=\rho_2\in(0,1]$, one unique stable positive equilibrium exists if $\mu_1\mu_2>1$, $R_1<0$, $H_3>0$ or $\mu_1\mu_2>1$, $R_1>0$, $H_3<0$, where $H_3\equiv R_3|_{\rho_2=\rho_1}$, i.e.,
\begin{align*}
\begin{autobreak}
H_3=
\rho_1^{3} \mu_1^{3} \mu_2^{3}
-4\,\rho_1^{3} \mu_1^{3} \mu_2^{2}
-4\,\rho_1^{3} \mu_1^{2} \mu_2^{3}
+17\,\rho_1^{3} \mu_1^{2} \mu_2^{2}
+4\,\rho_1^{3} \mu_1^{2} \mu_2
+4\,\rho_1^{3} \mu_1 \mu_2^{2}
-2\,\rho_1^{2} \mu_1^{2} \mu_2^{2}
-45\,\rho_1^{3} \mu_1 \mu_2
+8\,\rho_1^{2} \mu_1^{2} \mu_2
+8\,\rho_1^{2} \mu_1 \mu_2^{2}
-36\,\rho_1^{2} \mu_1 \mu_2
+27\,\rho_1^{3}
-4\,\rho_1 \mu_1 \mu_2
+54\,\rho_1^{2}
+36\,\rho_1
+8.
\end{autobreak}
\end{align*}
	
\end{theorem}

However, if $\rho_1\neq \rho_2$, the computations of searching for the algebraic description become particularly complicated. Furthermore, the geometric description can not be plotted either since more than 3 parameters are involved. We leave the exploration of this case for our future study.

\section{Snapback Repellers and Chaos}

According to observations through numerical simulations (see, e.g., \cite{Li2022C}), it seems to be supported that there exist chaotic dynamics in Kopel's map, which, however, is challenging to prove. In this regard, some studies (see, e.g., \cite{Canovas2018O, wu_complex_2010}) tried to derive the existence of chaos in Kopel's model by computing the topological entropy using float-point algorithms. However, the results based on float-point computations may not be reliable in some sense. In this section, we employ tools based on symbolic computations to prove the existence of snapback repellers in Kopel's map, which implies the occurrence of chaos in the sense of Li--Yorke according to Marotto’s theorem. In comparison, the results produced by symbolic computations are exact and error-free, which can be used to discover and prove mathematical theorems analytically. 

The approach of deriving the existence of chaos via Marotto’s theorem has been employed in a lot of literature such as \cite{ming_analysis_2021, ren_bifurcations_2017}, where, however, some details of proving snapback repellers are omitted. For example, in Section 5 of \cite{ren_bifurcations_2017}, the repelling neighborhood of the fixed point $z_0$ is obtained via numerical simulations and the relation $f^3(z_1)=z_0$ is verified via numerical calculations. These two critical steps are certainly not reliable and their computation details are not given by Ren et al.\ \cite{ren_bifurcations_2017}. Whereas, this section provides rigorous approach based on pure symbolic computations instead.

For the sake of simplicity, we only discuss the symmetric model of Kopel, which can be described by
\begin{equation}\label{eq:map-sym}
\begin{split}
\left\{\begin{aligned}
&x({t+1})=(1-\rho)x(t)+\rho\mu y(t)(1-y(t)),\\
&y(t+1)=(1-\rho)y(t)+\rho\mu x(t)(1-x(t)),
\end{aligned}\right.
\end{split}
\end{equation}
where $0< \rho\leq 1$ and $\mu>0$. As derived by Agiza \cite{agiza_analysis_1999}, there exist four equilibria in map \eqref{eq:map-sym}, of which the closed-form expressions are
\begin{align*}
	&E_0=(0,0),\\
	&E_1=\left(1-\frac{1}{\mu},1-\frac{1}{\mu}\right),\\
	&E_2=\left(\frac{\mu+1+\sqrt{(\mu+1)(\mu+3)}}{2\mu},\frac{\mu+1-\sqrt{(\mu+1)(\mu+3)}}{2\mu}\right),\\
	&E_3=\left(\frac{\mu+1-\sqrt{(\mu+1)(\mu+3)}}{2\mu},\frac{\mu+1+\sqrt{(\mu+1)(\mu+3)}}{2\mu}\right).
\end{align*}

\subsection{Preliminaries}
 We first describe the notion of snapback repeller and Marotto's theorem. Consider the following form of $n$-dimensional discrete dynamical systems
\begin{equation}\label{eq:di5-1}\nonumber
\begin{split}
\xvar(t+1)=F(\uvar,\xvar(t)),
\end{split}
\end{equation}
where $F:\mathbb{R}^n\rightarrow \mathbb{R}^n$ is a continuously differential map with parameters $\uvar \in \mathbb{R}^d$, where $\mathbb{R}$ is the field of real numbers. Define $B_{r}(\xvar)$ as the closed ball of radius $r>0$ centered at $\xvar$ under certain norm $\parallel \cdot\parallel$ in $\mathbb{R}^n$. We say that a point ${\bar\xvar}$ is a \emph{repeller} of ${F}$ with respect to the norm $\parallel\cdot\parallel$ if there exists a constant $s>1$ and a closed ball $B_{r}({\bar{\xvar}})$ such that
$$\parallel F(\xvar)-F(\yvar)\parallel> s\cdot\parallel \xvar-\yvar\parallel,$$ 
for any $\xvar, \yvar\in B_{r}({\bar{\xvar}})$ with $\xvar\neq \yvar$. In addition, we call $B_{r}(\bar\xvar)$ a \emph{repelling neighborhood} of $\bar\xvar$.

\begin{definition}\label{def:snapback}
We say that a fixed point $\xvar^*$ is an \emph{m-snapback repeller} of ${F}$, where $m$ is an integer greater than 1, if 
\begin{enumerate}
	\item ${\xvar^*}$ is a repeller of ${F}$,
	\item $B_{r}({\xvar^*})$ is a repelling neighborhood of ${\xvar^*}$, and at least one point $\xvar_{0}\in B_{r}({\xvar^*})$ can be found such that
	\item $\xvar_{0}\neq {\xvar^*}$, $\xvar_{m}={\xvar^*}$,
and $|{F}'(\xvar_k)|\neq0$ for $1\leq k\leq m$, where $\xvar_{k}={F}^{k}(\xvar_{0})$.
\end{enumerate}
\end{definition}

Under this definition, the following lemma (Marotto's theorem) holds, which was initially proved by Marotto \cite{marotto_snap-back_1978, marotto_redefining_2005}.

%
%
%
%

\begin{lemma}[Marotto's theorem]\label{thm:li-chaos}

If $F$ possesses a snapback repeller, then $F$ is chaotic in the sense of Li--Yorke \cite{li_period_1975}, namely the following two statements are true.
\begin{enumerate}
	\item For any integer $k$, there is a point $\pvar_k\in I$ with period $k$, i.e., $F^k(\pvar_k)=\pvar_k$, and $F^i(\pvar_k)\neq \pvar_k$ for $1\leq i <k$.
	
	\item There is an uncountable set $S$ (containing no periodic points), which satisfies the following conditions:
	\begin{enumerate}
		\item for any $\xvar,\yvar\in S$ with $\xvar\neq \yvar$,
		\begin{equation}\label{eq:away2p}
			\limsup_{n\rightarrow\infty}\parallel F^n(\xvar)-F^n(\yvar)\parallel>0,
		\end{equation}
				and
		\begin{equation}\label{eq:near2p}
			\liminf_{n\rightarrow\infty}\parallel F^n(\xvar)-F^n(\yvar)\parallel=0;
		\end{equation}
		$$$$		
		\item for every point $\xvar\in S$ and every periodic point $\pvar\in I$,
		\begin{equation}\label{eq:awaycircle}
			\limsup_{n\rightarrow\infty}\parallel F^n(\xvar)-F^n(\pvar)\parallel>0.
		\end{equation}

	\end{enumerate}
\end{enumerate}
\end{lemma}

\begin{remark}
Eq.\ \eqref{eq:near2p} means that every trajectory in $S$ can wander arbitrarily close to every other. 
However, Eq.\ \eqref{eq:away2p} means that no matter how close two distinct trajectories in $S$ may come to each other, they must eventually wander away.
Furthermore, by \eqref{eq:awaycircle} it is obtained that every trajectory in $S$ goes away from any periodic orbit of $F$. 
\end{remark}

Marotto's theorem is significant in extending the analytic theory of chaos from one-dimensional maps to multi-dimensional maps. It is also effective in applications, for example, in finding the chaotic regimes (parameter regions) for dynamical systems. From a practical perspective, there exist two directions to determine whether a fixed point is a snapback repeller for a multi-dimensional map. The first one is to find a repelling neighborhood $B_r(\xvar^*)$ of a fixed point $\xvar^*$ and a preimage point $\xvar_0$ of $\xvar^*$ lying in $B_r(\xvar^*)$, i.e., with 
$$F^{m}(\xvar_0)=\xvar^*,~\xvar_{0}\in B_r(\xvar^*),~\xvar_{0}\neq \xvar^*,~\text{for some}~m>1.$$ 
The second direction is to construct the preimages $\left\{\bar\xvar^{-k}\right\}_{k=1}^{\infty}$ of a repeller $\bar\xvar$ such that $\lim\limits_{k\rightarrow\infty}F(\bar\xvar^{-k})=\bar\xvar$,
where
$F(\bar\xvar^{-k})=\bar\xvar^{-k+1}~\text{with}~k\geq2,~\text{and}~F(\bar\xvar^{-1})=\bar\xvar$ (see, e.g., \cite{gardini_critical_2011,liao_snapback_2012}). We call such an orbit $\left\{\bar\xvar^{-k}\right\}_{k=1}^{\infty}$ a (degenerate) homoclinic orbit for the repeller $\bar\xvar$. The existence of homoclinic orbits guarantees the existence of snapback repellers in the repelling neighborhood of $\bar\xvar$. Marotto's theorem thus holds without knowing the repelling neighborhood of a fixed point.

%


In this paper, we adopt the first direction to establish the existence of snapback repellers. Therefore, estimating the repelling neighborhood for a repeller is the key step in utilizing Marotto's theorem. Moreover, a computable norm is needed for practical computations. Previous results based on the application of Marotto's theorem in the literature were almost considered under the Euclidean norm. Below, we list several techniques used to determine whether a repelling fixed point is a snapback repeller. A few studies combining mathematical analysis with complex manual processing may be found in \cite{jing_bifurcation_2004, wang_bifurcation_2010, zhao_bifurcation_2021}. In addition, numerical computation techniques have been used to prove the existence of snapback repellers (see, e.g., \cite{peng_numerical_2007, liao_snapback_2012}). On the other hand, confirming that some preimage of a repelling fixed point lies in the repelling neighborhood of this fixed point is a nontrivial task. Notably, Liao and Shih \cite{liao_snapback_2012} gave an estimation of the radius of the repelling neighborhood for a repelling fixed point. This estimation is of essential and practical significance to numerical computations of snapback points (see \cite[Section 2]{liao_snapback_2012}). Note that all the examples presented in the aforementioned work, the radius of repelling neighborhood depends heavily on the parameter values. In this paper, we will analyze snapback repellers of discrete dynamical systems automatically by using symbolic computation methods. More precisely, our method can produce parameter sample points under which a snapback repeller may exist. Then, we employ the trial and error correction method to select parameter points, where a snapback repeller indeed exists. Remark that a prior radius size of the repelling neighborhood may be given to make the selected parameter points more sufficient. See more details in Section 4.3.

\subsection{Might (0,0) be a snapback repeller?}

We start by considering the third condition of Definition \ref{def:snapback}, which can be characterized by the following semi-algebraic system with equations and inequations.

\begin{equation}\label{eq:iter-m}
\left\{\begin{split}
&{F}^{m}({\xvar}_0)-{{\xvar^*}}=0,\\
&{\xvar}_0\neq {{\xvar^*}},\\
&|{F}'({\xvar}_k)|\neq0,~ k=1,\ldots,m,
\end{split}\right.
\end{equation}
where $F'$ is the Jacobian matrix of $F$ and $|\cdot|$ represents the determinant. One can see that the Jacobian matrix of map \eqref{eq:map-sym} evaluated at a fixed point $\xvar^*=(x^*,y^*)$ is given by
\begin{equation}\label{eq:jab-sym}
\begin{split}
M(\xvar^*)=\left[\begin{array}{cc}
    1-\rho & \rho\mu(1-2\, y^*)\\
    \rho\mu(1-2\, x^*) & 1-\rho\\
  \end{array}\right]. 
\end{split}
\end{equation}

As an example, we only consider the possibility of $E_0$ to be a 2-snapback repeller due to space limitations. Other equilibria and other values of $m$ can be similarly discussed. We should emphasize that the computational complexity of our approach increases rapidly as $m$ grows. 

Assume that $\xvar_0=(x_0,y_0)$. Then, we have 
$$\xvar_1=(x_1,y_1)=((1-\rho)x_{0}+\rho\mu y_{0}(1-y_0),(1-\rho)y_{0}+\rho\mu x_{0}(1-x_0)),$$
and
$$\xvar_2=(x_2,y_2)=((1-\rho)x_{1}+\rho\mu y_{1}(1-y_1),(1-\rho)y_{1}+\rho\mu x_{1}(1-x_1)).$$
It is known that $\xvar_2=E_0=(0,0)$ since 2-snapback repellers are discussed. Consequently, system \eqref{eq:iter-m} can be written as
\begin{equation}\label{eq:iter}
\left\{
\begin{split}
&P_1\equiv \big(-x_{{0}}^{4}+2x_{{0}}^{3}-x_{{0}}^{2}\big)\mu^{3
}\rho^{3}+\big(-2x_{{0}}^{2}y_{{0}}+2x_{{0}}y_{{0}}\big)\mu^{2}\rho^{3}
+\big(2x_{{0}}^{2}y_{{0}}-x_{{0}}^{2}-2x_{{0}}y_{{0}}+x_{{0}}\big)\mu^{2}\rho^{2}\\
&\qquad-\mu\rho^{3}y_{{0}}^{2}+\big(3y_{{0}}^{2}-2y_{{0}}\big)\mu\rho^{2}
+\big(-2y_{{0}}^{2}+2y_{{0}}\big)\rho\mu+x_{{0}}\rho^{2}-2x_{{0}}\rho+x_{{0}}- 0
=0,\\
&P_2\equiv \left(-y_{{0}}^{4}+2y_{{0}}^{3}-y_{{0}}^{2}\right)\mu^{3
}\rho^{3}+\left(-2x_{{0}}y_{{0}}^{2}+2x_{{0}}y_{{0}}\right)\mu^{2}\rho^{3}
+\left(2x_{{0}}y_{{0}}^{2}-2x_{{0}}y_{{0}}-y_{{0}}^{2}+y_{{0}}\right)\mu^{2}\rho^{2}\\
&\qquad-\mu\rho^{3}x_{{0}}^{2}+\left(3x_{{0}}^{2}-2x_{{0}}\right)\mu
\rho^{2}+\left(-2x_{{0}}^{2}+2x_{{0}}\right)\rho\mu+y_{{0}}\rho^{2}
-2y_{{0}}\rho+y_{{0}}- 0=0,\\
&( 0-x_0)^2+( 0-y_0)^2\neq 0,\\
&|M((1-\rho)x_{0}+\rho\mu y_{0}(1-y_0),(1-\rho)y_{0}+\rho\mu x_{0}(1-x_0))|\neq0,~|M(0, 0)|\neq0,
\end{split}
\right.
\end{equation}
where $M$ is the Jacobian matrix of Kopel's map given in \eqref{eq:jab-sym}. One can see that this is a semi-algebraic system with variables $x_0,y_0$ and parameters $\rho,\mu$.

From a computational point of view, it is not practical to verify the conditions for the original definition of repellers. However, the following lemma, initially derived by Li and Chen \cite{li_improved_2003}, gives a computational method for identifying a repeller under the Euclidean norm.

\begin{lemma}\label{lem:norm2}
Let $\parallel\cdot\parallel_2$ be the Euclidean norm. Suppose that $\xvar^*$ is a fixed point of $F$ and $F$ is continuously differentiable in some closed ball $B_r(\xvar^*)$.
If all eigenvalues of $(F'({\xvar^*}))^T F'({\xvar^*})$ are greater than $1$, then there is $s>1$ and $r'\in (0,r]$ such that 
$$\parallel F(\xvar)-F(\yvar) \parallel_2 > s\cdot \parallel \xvar-\yvar \parallel_2,$$
for all $\xvar,\yvar\in B_{r'}(\xvar^*)$ with $\xvar\neq \yvar$, and all eigenvalues of $(F'(\xvar))^T F'(\xvar)$ exceed $1$ for any $\xvar\in B_{r'}( \xvar^*)$.
\end{lemma}

One of the main contributions of the work \cite{huang_analysis_2019} by Huang and Niu is providing an algebraic criterion for identifying whether all eigenvalues of a given characteristic polynomial lie outside the unit circle. In the sequel, we recall this algebraic criterion. Readers can refer to \cite{bistritz_zero_1984} for additional details. Suppose that the characteristic polynomial is of the form
\begin{equation}\nonumber
\begin{split}
D(\lambda)=d_0+d_1\lambda+\ldots+d_n\lambda^n,
\end{split}
\end{equation}
where $d_i=d_i(\uvar,\xvar)$, $i=0,\ldots,n$. Then, denote by $D^{*}(\lambda)$ the reciprocated polynomial of $D(\lambda)$, namely 
$$D^{*}(\lambda)=\lambda^{n}D(\lambda^{-1})=d_{n}+d_{n-1}\lambda+\cdots+d_{0}\lambda^{n}.$$

For the polynomial $D(\lambda)$, we can assign to it a so-called \emph{zero discriminant sequence} of $n+1$ polynomials $T_{n}(\lambda),T_{n-1}(\lambda),\cdots,T_{0}(\lambda)$ according to the following recursion formula.
\begin{equation}\label{eq:T_k}
\begin{split}
&T_{n}(\lambda)=D(\lambda)+D^{*}(\lambda),\\
&T_{n-1}(\lambda)=[D(\lambda)-D^{*}(\lambda)]/(\lambda-1),\\
&T_{k-2}(\lambda)=\lambda^{-1}[\delta_{k}(\lambda+1)T_{k-1}(\lambda)-T_{k}(\lambda)],~~k=n,n-1,\ldots,2,
\end{split}
\end{equation}
where $\delta_{k}=T_{k}(0)/T_{k-1}(0)$. It is evident that the recursive construction requires the normal conditions 
$$T_{n-i}(0)\neq0,~~ i=0,1,\ldots,n.$$ 
The construction is interrupted when $T_{k}(0)=0$ occurs for some $k$. As stated in \cite{bistritz_zero_1984}, such singular cases can be classified into two types. The following lemma is one of the main results of \cite{huang_analysis_2019}, which shows that there is no need to consider the singular cases.
\begin{lemma}\label{lem:var}
All zeros of $D(\lambda)$ lie outside the unit circle if and only if the normal conditions $T_{n-i}(0)\neq0,~ i=0,1,\ldots,n$ hold and $\mbox{Var}(T_n(1),\ldots,T_0(1))=n$, where $\mbox{Var}(T_n(1),\ldots,T_0(1))$ denotes the number of sign variations of the sequence $T_n(1),\ldots,T_0(1)$.
\end{lemma}

\begin{remark}
	The condition $\mbox{Var}(T_n(1),\ldots,T_0(1))=n$ includes two possible cases:
$$T_{n}(1)>0,T_{n-1}<0,\ldots,~~\text{or}~~T_{n}(1)<0,T_{n-1}>0,\ldots,$$
which can be expressed uniformly as
$$(-1)^{i+j-1}T_{n-i}(1)>0,~~i=0,\ldots,n,~~\text{for}~j=1~\text{or}~j=2.$$
\end{remark}

Returning to Kopel's model, the characteristic polynomial of the matrix $(M(\xvar))^T M(\xvar)$ can be written as
\begin{align*}
	\begin{autobreak}
D(\lambda)=
\lambda^{2}
+(
-2
+(
-4x^{2}
-4y^{2}
+4x
+4y
-2)\mu^{2}\rho^{2}
-2\rho^{2}
+4\rho)\lambda 
+(16x^{2}y^{2}
-16x^{2}y 
-16xy^{2}
+4x^{2}
+16xy
+4y^{2}
-4x
-4y
+1)\mu^{4}\rho^{4}
+( 
-8xy
+4x
+4y
-2)\mu^{2}\rho^{4} 
+(16xy
-8x
-8y
+4)\mu^{2}\rho^{3}
+(
-8xy
+4x
+4y
-2)\mu^ {2}\rho^{2}
+\rho^{4}
-4\rho^{3}
+6\rho^{2}
-4\rho
+1.  	
	\end{autobreak}
\end{align*}
The reciprocated polynomial of $D(\lambda)$ is
\begin{align*}
	\begin{autobreak}
D^{*}(\lambda)=
((16x^{2}y^{2}
-16x^{2}y
-16xy^{2}
+4x^{2}
+16xy
+4y^{2}
-4x
-4y
+1)\mu^{4}\rho^{4} 
+(
-8xy
+4x
+4y
-2)\mu^{2}\rho^{4}
+(16xy
-8x
-8y
+4)\mu^{2}\rho^{3}
+(
-8xy
+4x
+4y
-2)\mu^ {2}\rho^{2} 
+\rho^{4}
-4\rho^{3}
+6\rho^{2}
-4\rho
+1)\lambda^{2} 
+(
-2
+(
-4x^{2}
-4y^{2}
+4x
+4y
-2)\mu^{2}\rho^{2}
-2\rho^{2}
+4\rho)\lambda
+1. 	
	\end{autobreak}
\end{align*}
Then, we acquire the zero discriminant sequence
\begin{equation}\label{eq:K-5}\nonumber
\begin{split}
T_{2}(\lambda)=t_{2,2}(\lambda^{2}+1)+t_{2,1}\lambda,~~
T_{1}(\lambda)=t_{1,1}(\lambda+1),~~ T_{0}(\lambda)=t_{0,0},
\end{split}
\end{equation}
where
\begin{align*}
	\begin{autobreak}
t_{2,2}=
2
+(16x^{2}y^{2}
-16x^{2}y
-16xy^{2}
+4x^{2}
+16xy
+4y^{2}
-4x
-4y
+1)\mu^{4}\rho^{4} 
+(
-8xy
+4x
+4y
-2)\mu^{2}\rho^{4}
+(16xy
-8x
-8y
+4)\mu^{2}\rho^{3} 
+(
-8xy
+4x
+4y
-2)\mu^{2}\rho^{2} 
+\rho^{4}
-4\rho^{3}
+6\rho^{2}
-4\rho, 	
	\end{autobreak}\\
	\begin{autobreak}
t_{2,1}=
-4
+(
-8x^{2}
-8y^{2}
+8x
+8y
-4)\mu^{2}\rho^{2}
-4\rho^{2}
+8\rho,
	\end{autobreak}\\
	\begin{autobreak}
t_{1,1}=
-(2
+(4xy
-2x
-2y
+1)\mu^{2}\rho
-\rho)(
-2
+(4xy
-2x
-2y
+1)\mu^{2}\rho^{2} 
-\rho^{2}
+2\rho)\rho,		
	\end{autobreak}\\
	\begin{autobreak}
t_{0,0}=
8
+(32x^{2}y^{2}
-32x^{2}y
-32xy^{2}
+8x^{2}
+32xy
+8y^{2}
-8x
-8y
+2)\mu^{4}\rho^{4} 
+(
-16xy
+8x
+8y
-4)\mu^{2}\rho^{4}
+(32xy
-16x
-16y
+8 )\mu^{2}\rho^{3}
+(8x^{2}
-16xy
+8y^{2})\mu^{2}\rho^{2} 
+2\rho^{4}
-8\rho^{3}
+16\rho^{2}
-16\rho.	
	\end{autobreak}
\end{align*}
%
Consequently, we have that
\begin{align*}
&T_{2}(1)|_{E_0}=2 \rho^{2} \left(\mu-1\right) \left(\mu+1\right) \left(\rho \mu-\rho+2\right) \left(\rho \mu+\rho-2\right),\\
&T_{1}(1)|_{E_0}=-2\rho \left(\rho^{2} \mu^{2}-\rho^{2}+2 \rho-2\right) \left(\rho \mu^{2}-\rho+2\right),\\
&T_{0}(1)|_{E_0}=2\left(\rho^{2} \mu^{2}+2 \rho^{2} \mu+\rho^{2}-2 \rho \mu-2 \rho+2\right) \left(\rho^{2} \mu^{2}-2 \rho^{2} \mu+\rho^{2}+2 \rho \mu-2 \rho+2\right).
\end{align*}
According to Lemmas \ref{lem:norm2} and \ref{lem:var}, the equilibrium $E_0=(0,0)$ is a repeller (satisfies the first condition of Definition \ref{def:snapback}) if 
\begin{equation}\label{eq:TE0}
T_{2}(1)|_{E_0}>0,~
-T_{1}(1)|_{E_0}>0,~
T_{0}(1)|_{E_0}>0,~~
T_{2}(0)|_{E_0}\neq0,~T_{1}(0)|_{E_0}\neq0,~
T_{0}(0)|_{E_0}\neq0,
\end{equation}
or
\begin{equation}\label{eq:TE0-2}
-T_{2}(1)|_{E_0}>0,~
T_{1}(1)|_{E_0}>0,~
-T_{0}(1)|_{E_0}>0,~~
T_{2}(0)|_{E_0}\neq0,~
T_{1}(0)|_{E_0}\neq0,~
T_{0}(0)|_{E_0}\neq0.
\end{equation}
By our numerical simulations, chaos may take place if $\mu$ is large enough. Moreover, one can see that $T_{2}(1)|_{E_0}>0$ if $\mu$ is large enough. Therefore, Eq.\ \eqref{eq:TE0} is more likely to be fulfilled than \eqref{eq:TE0-2}. Below, we only verify the possibility of \eqref{eq:TE0}. From the second condition of Definition \ref{def:snapback}, one knows that $\xvar_0$ is selected from a repelling neighborhood of $E_0$. Thus, $\xvar_0$ should also be a repeller and the following conditions must be satisfied.
\begin{equation}\label{eq:Tx0}
T_{2}(1)|_{\xvar_0}>0,~
-T_{1}(1)|_{\xvar_0}>0,~
T_{0}(1)|_{\xvar_0}>0,~~
T_{2}(0)|_{\xvar_0}\neq0,~
T_{1}(0)|_{\xvar_0}\neq0,~
T_{0}(0)|_{\xvar_0}\neq0.	
\end{equation}

According to Definition \ref{def:snapback}, we need to consider the semi-algebraic system formed by combining Eqs.\ \eqref{eq:iter}, \eqref{eq:TE0}, \eqref{eq:Tx0}, and then detect whether this system possesses at least one real solution. In general, a semi-algebraic system is of the form
\begin{equation}\label{semi-alg}
\left\{
\begin{array}{l}
P_1(u_1, \ldots, u_d,x_1, \ldots, x_n)=0,\\
\qquad\qquad\qquad\vdots \\
P_n(u_1, \ldots, u_d, x_1, \ldots, x_n)=0,\vspace{4pt}\\
Q_1(u_1, \ldots, u_d,x_1, \ldots, x_n)\lessgtr 0,\\
\qquad\qquad\qquad\vdots \\
Q_r(u_1, \ldots, u_d, x_1, \ldots, x_n)\lessgtr 0,
\end{array}
\right.
\end{equation}
where the symbol $\lessgtr$ stands for any of $>$, $\geq$, $<$, $\leq$, and $\neq$, and $P_i, Q_j$ are polynomials over the field of real numbers, with $u_1, \ldots, u_d$ as their parameters and $x_1, \ldots, x_n$ as their variables. Based on the triangular decomposition and resultant methods, we can compute the so-called \emph{border polynomial} of system \eqref{semi-alg}. One of the properties of a border polynomial is that its zeros divide the parameter space into several connected regions and the number of real solutions of \eqref{semi-alg} can not change if we vary the parameter point inside the same region. Readers can refer to \cite{Li2014C, Yang2001A} for more information regarding border polynomials. Actually, the basic idea of computing border polynomials has been applied in Section 3 to establish conditions for distinct number of stable equilibria.

Herein, the resulting border polynomial $BP$ of system \eqref{eq:iter}+\eqref{eq:TE0}+\eqref{eq:Tx0} is extremely complicated and takes more than two pages to present. In Figure \ref{fig:bpsnapback}, we use the Maple function ``implicitplot'' to depict the curve of $BP=0$ approximately. Accordingly, one can understand how complex this border polynomial could be. 


\begin{figure}[htbp]
  \centering
  \subfloat[$(\rho,\mu)\in (0,1)\times(0,10)$]{\includegraphics[width=0.45\textwidth]{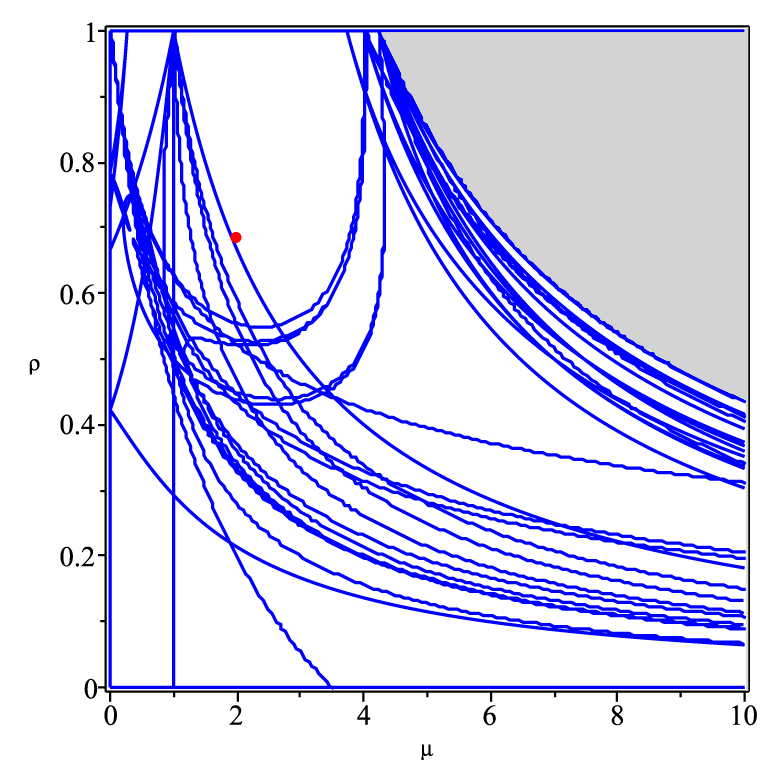}} 
  \subfloat[$(\rho,\mu)\in (0,1)\times(10,60)$]{\includegraphics[width=0.45\textwidth]{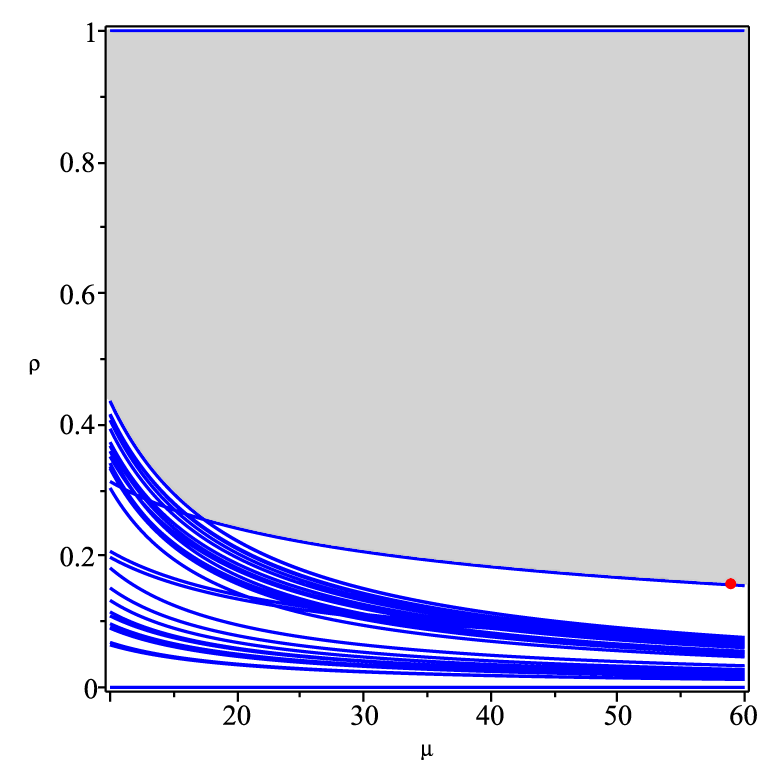}} \\
  \caption{Zeros of the border polynomial of the system \eqref{eq:iter}+\eqref{eq:TE0}+\eqref{eq:Tx0}.}
\label{fig:bpsnapback}
\end{figure}
As shown by Figure \ref{fig:bpsnapback}, the curve of $BP=0$ divides the parameter plane into many connected regions. According to the properties of border polynomials, the number of real solutions will not change if we vary the parameter point inside the same region. Because of the complexity of $BP$, the selection of sample points is thus complicated, which can be automated by using the PCAD method. Based on our computations, 4324 sample points are selected, of which there are 1845 sample points where system \eqref{eq:iter}+\eqref{eq:TE0}+\eqref{eq:Tx0} has at least one real solution. We list nine of these 1845 sample points in Table \ref{tab:sample}.

\begin{table}[htbp]
	\centering 	
\caption{Some sample points where \eqref{eq:iter}+\eqref{eq:TE0}+\eqref{eq:Tx0} has at least one real solution}
\label{tab:some-sample} 
\begin{tabular}{|l||l||l|}
  \hline
$\rho=175/256, \mu=2$ & $\rho=5/32, \mu=59$ & $\rho=1/8, \mu=16$\\
  \hline
$\rho=9/32, \mu=7$ & $\rho=7/8, \mu=2$ & $\rho=1537/2048, \mu=6$\\
  \hline
$\rho=7/16, \mu=29/8$ &$\rho=41/64, \mu=3$ &$\rho=15/16, \mu=2$\\
  \hline
\end{tabular}
\end{table}

We should mention that the main theorem (Theorem 4.2) of \cite{huang_analysis_2019} that provides conditions for the existence of snapback repellers is insufficient. That is to say, for any parameter point in the same region where any of the aforementioned 1845 sample points lies, it is probable but not guaranteed that a snapback repeller exists. For example, in Section \ref{sec:validation} readers will see that at the first sample point $(\rho,\mu)=(175/256,2)$ listed in Table \ref{tab:some-sample}, $E_0$ is not a 2-snapback repeller, which means that the conditions of \cite[Theorem 4.2]{huang_analysis_2019} are insufficient. Additional steps are needed to further validate the existence of snapback repellers, which we illustrate in Section \ref{sec:validation}.

\subsection{Final validation and formal proof}\label{sec:validation}

Although any solution $\xvar_0$ of \eqref{eq:iter}+\eqref{eq:TE0}+\eqref{eq:Tx0} is a repeller and satisfies the third condition of Definition \ref{def:snapback}, the second condition of Definition \ref{def:snapback} may not be fulfilled. That is to say, we may not find a repelling neighborhood $B_r(E_0)$ such that $\xvar_0\in B_r(E_0)$. For example, when $\rho=175/256$ and $\mu=2$, we have
\begin{align*}
P_1=&-\frac{5359375}{2097152} x_0^{4}+\frac{5359375}{1048576} x_0^{3}+\frac{2480625}{2097152} y_0 x_0^{2}-\frac{9279375}{2097152} x_0^{2}-\frac{2480625}{2097152} y_0 x_0\\
&-\frac{4776975}{8388608} y_0^{2}+\frac{129061}{65536} x_0+\frac{14175}{16384} y_0,\\
P_2=&-\frac{5359375}{2097152} y_0^{4}+\frac{2480625}{2097152} x_0 y_0^{2}+\frac{5359375}{1048576} y_0^{3}-\frac{4776975}{8388608} x_0^{2}-\frac{2480625}{2097152} y_0 x_0\\
&-\frac{9279375}{2097152} y_0^{2}+\frac{14175}{16384} x_0+\frac{129061}{65536} y_0,
\end{align*}
which can be approximately solved by
$$(0,0),~(1.1667,1.1667),~(-0.165, 0.966),~(0.940, -0.142).$$
It is evident that at $(x_0,y_0)=(0,0)$, the condition $( 0-x_0)^2+( 0-y_0)^2\neq 0$ in \eqref{eq:iter} is not fulfilled. However, the other three solutions satisfy \eqref{eq:iter}+\eqref{eq:TE0}+\eqref{eq:Tx0}. 

Figure \ref{fig:no-snapback} depicts these three non-zero solutions. The considered equilibrium $E_0$ is colored in red, while the non-zero solutions are colored in blue. The regions where $T_2(1)>0$, $T_1(1)<0$, and $T_0(1)>0$ holds are marked in yellow. Furthermore, we mark the curves $T_2(1)=0$, $T_1(1)=0$, and $T_0(1)=0$ in red, green and blue, respectively. However, the blue curve does not appear in the plotted window $[-0.5,1.5]\times [-0.5,1.5]$. One can observe that it is impossible to find a closed ball $B_r(E_0)$ such that every point in $B_r(E_0)$ is a repeller ($B_r(E_0)$ lies in the yellow regions) and one of the blue points is in $B_r(E_0)$. Consequently, no snapback repellers exist when $\rho=175/256$ and $\mu=2$.

\begin{figure}[htbp]
    \centering
    \includegraphics[width=0.49\textwidth]{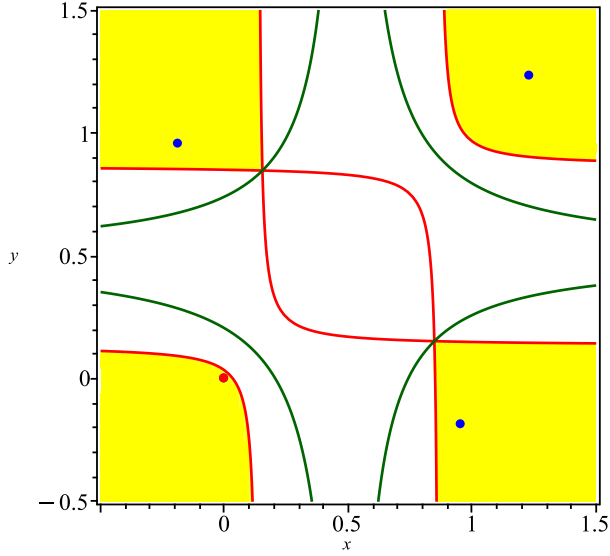}
    \caption{The (yellow) region of repellers when $\rho=175/256$ and $\mu=2$. The fixed point $E_0$ is marked in red, while the corresponding points of $\xvar_0$ for $E_0$ are marked in blue.}
    \label{fig:no-snapback}
\end{figure}

In comparison, we consider $\rho = 5/32$ and $\mu = 59$. It is obtained that
\begin{align*}
P_1=& -\frac{25672375}{32768} x_0^{4}+\frac{25672375}{16384} x_0^{3}+\frac{2349675}{16384} y_0 \,x_0^{2}-\frac{28457175}{32768} x_0^{2}-\frac{2349675}{16384} y_0 x_0\\
&-\frac{469935}{32768} y_0^{2}+\frac{43877}{512} x_0+\frac{7965}{512} y_0,\\
P_2=& -\frac{25672375}{32768} y_0^{4}+\frac{2349675}{16384} x_0 \,y_0^{2}+\frac{25672375}{16384} y_0^{3}-\frac{469935}{32768} x_0^{2}-\frac{2349675}{16384} y_0 x_0\\
&-\frac{28457175}{32768} y_0^{2}+\frac{7965}{512} x_0+\frac{43877}{512} y_0,
\end{align*}
which can be solved by the following float-point solutions using numerical algorithms.
\begin{equation}\label{eq:x0-float}
\begin{split}
&(0,0),~(1.092,1.091),~(-0.082,0.867),~(-0.020,0.125),~\\
&(0.015,1.010),~(0.125,-0.020),~(0.867,-0.082),~(0.992,1.091),~\\
&(1.010,0.015),~(1.092,0.992),~(0.041,0.868),~(0.868,0.041),~\\
&(-0.084,0.992),~(0.992,-0.084),~(0.122,0.122),~(0.969,0.969).	
\end{split}
\end{equation}

Figure \ref{fig:snapback} presents these solutions. The equilibrium $E_0=(0,0)$ is the red point, while the other non-zero solutions are the blue ones. In addition, the regions where $T_2(1)>0$, $T_1(1)<0$, and $T_0(1)>0$ hold are also colored in yellow. Therefore, we could construct the closed ball $B_r(E_0)$ with $r=1/5$, i.e., $x^2+y^2\leq (1/5)^2$. In Figure \ref{fig:snapback}, we plot the edge of this ball with a purple circle. Then, any of the blue points inside the circle can be selected as the $\xvar_0$ that satisfies the conditions of Definition \ref{def:snapback}.

\begin{figure}[htbp]
    \centering
    \includegraphics[width=0.49\textwidth]{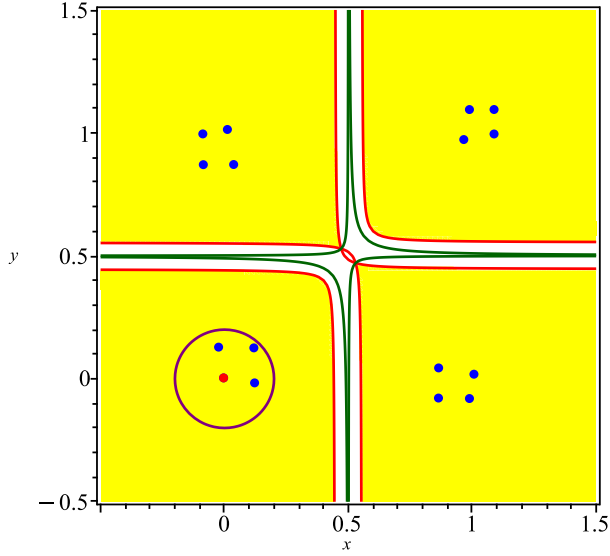}
    \caption{The (yellow) region of repellers when $\rho = 5/32$ and $\mu = 59$. The fixed point $E_0$ is marked in red, while the corresponding points of $\xvar_0$ for $E_0$ are marked in blue.}
    \label{fig:snapback}
\end{figure}

However, the above observations are based on numerical plotting, which may be unreliable. Furthermore, the float-point solutions \eqref{eq:x0-float} produced by numerical algorithms may not be reliable either. In the sequel, we would like to prove the existence of repellers rigorously by replacing those numerical steps with symbolic ones.

\begin{theorem}\label{thm:snapback}
Let $\rho = 5/32$ and $\mu = 59$. Then, $E_0=(0,0)$ is a snapback repeller of map \eqref{eq:map-sym}, which proves that chaotic dynamics exist in the sense of Li--Yorke.
\end{theorem}

\begin{proof}
The first condition of Definition \ref{def:snapback} is fulfilled since $\rho = 5/32$ and $\mu = 59$ are acquired by assuming \eqref{eq:TE0}. According to Lemma \ref{lem:norm2}, $E_0$ is a repeller under the Euclidean norm. The third condition of Definition \ref{def:snapback} is obvious because $\xvar_0$ is computed from \eqref{eq:iter}.

To validate the second condition of Definition \ref{def:snapback}, we need to find a closed ball that is a repelling neighborhood of $E_0$, and then prove that one of the solutions of system \eqref{eq:iter}+\eqref{eq:TE0}+\eqref{eq:Tx0}, i.e., $\xvar_0$ lies in this ball. Denote $C_i=(-1)^{i}T_i(1)|_{\rho=5/32,\mu=59}$, $i=2,1,0$. Then,
\begin{align*}
C_2=&\frac{7573350625}{32768} x^{2} y^{2}-\frac{7573350625}{32768} x^{2} y-\frac{7573350625}{32768} x \,y^{2}+\frac{7484237025}{131072} x^{2}+\frac{15083260025}{65536} x y\\
&+\frac{7484237025}{131072} y^{2}-\frac{927599475}{16384} x-\frac{927599475}{16384} y+\frac{113567625}{8192},\\
C_1=&\frac{7573350625}{32768} x^{2} y^{2}-\frac{7573350625}{32768} x^{2} y-\frac{7573350625}{32768} x \,y^{2}+\frac{7573350625}{131072} x^{2}+\frac{15083260025}{65536} x y\\
&+\frac{7573350625}{131072} y^{2}-\frac{938738675}{16384} x-\frac{938738675}{16384} y+\frac{116342985}{8192},\\
C_0=&\frac{7573350625}{32768} x^{2} y^{2}-\frac{7573350625}{32768} x^{2} y-\frac{7573350625}{32768} x \,y^{2}+\frac{7662464225}{131072} x^{2}+\frac{15083260025}{65536} x y\\
&+\frac{7662464225}{131072} y^{2}-\frac{949877875}{16384} x-\frac{949877875}{16384} y+\frac{119183881}{8192}.
\end{align*}
If all points in the closed ball $B_r(E_0)=x^2+y^2\leq (1/5)^2$ satisfy $C_i>0$, $i=2,1,0$, then it can be proved that all eigenvalues of $M(\xvar)^T M(\xvar)$ ($M$ is the Jacobian matrix of map \eqref{eq:map-sym}) are greater than 1 at any point $\xvar\in B_r(E_0)$. The PCAD method permits us to select at least one sample point from each region defined by $x^2+y^2\leq  (1/5)^2$ and $C_i\neq 0$, $i=2,1,0$. We obtain one unique sample point $(0,0)$, which implies that any of $C_i= 0$, $i=2,1,0$, does not intersect with $B_r(E_0)$. It is evident that $C_i>0$ at $(0,0)$ for all $i=2,1,0$, which means that all points in $B_r(E_0)$ satisfy $C_i>0$, $i=2,1,0$. In consequence, all eigenvalues of $M(\xvar)^T M(\xvar)$ are greater than 1 for $x^2+y^2\leq  (1/5)^2$.

The real root isolation method (see, e.g., the modified Uspensky algorithm \cite{Collins1983R}) permits us to isolate each real solution of a semi-algebraic system with a box. Moreover, the size of boxes can be made as small as we want. Accordingly, the three blue points $(x_0,y_0)$ inside the purple circle in Figure \ref{fig:snapback} can be isolated by
\begin{align*}
&(x_0,y_0)\in \left[ -\frac{5279}{262144},-\frac{84463}{4194304}\right]\times \left[ \frac{1025}{8192},\frac{2051}{16384}\right],\\	
&(x_0,y_0)\in \left[\frac{8201}{65536},\frac{65609}{524288} \right]\times \left[-\frac{83}{4096},-\frac{41}{2048} \right],\\	
&(x_0,y_0)\in \left[\frac{128075}{1048576},\frac{32019}{262144} \right]\times \left[\frac{2001}{16384},\frac{1001}{8192} \right].
\end{align*}
It can be verified that all corner points of these boxes satisfy $x^2+y^2\leq (1/5)^2$. That is to say, the three boxes and thus the three points lie inside $B_r(E_0)$. 

In conclusion, the equilibrium $E_0=(0,0)$ is a snapback repeller of map \eqref{eq:map-sym} when $\rho = 5/32$ and $\mu = 59$. Therefore, the existence of chaos in the sense of Li--Yorke follows according to Lemma \ref{thm:li-chaos}, which completes the proof.
\end{proof}

\begin{remark}
We remark that the selected closed ball $B_r(E_0)$ with $r=1/5$ does provide a repelling neighborhood of the repelling fixed point $E_0$. Specifically, for the neighborhood $B_r(E_0)$, one can verify that $s_1=67/8$, $\eta_{r}=59/16$, where the definitions of $s_1$ and $\eta_{r}$ can be found in \cite{liao_snapback_2012}. Then, we have $s_1-\eta_{r}=75/16>1$, which proves that $B_r(E_0)$ is indeed a repelling neighborhood according to \cite[Proposition 2.2]{liao_snapback_2012}. It should be noted that all computations in the above proof are essentially based on rational rather than float-point numbers. Mathematical objects such as rational polynomials and rational intervals are handled. In this regard, readers may understand that the involved computations are symbolic and the obtained results are exact (error-free). Therefore, tools based on symbolic computations can replace pen and paper to prove mathematical theorems in some sense.
\end{remark}

For other parameter points in the region where $(\rho,\mu)=(5/32,59)$ lies (see the light-gray region in Figure \ref{fig:bpsnapback}), we also need to repeat the validation process. We randomly choose more than 100 parameter points in the light-gray region. The computational results show that $E_0$ is a snapback repeller at all these points. In Figure \ref{fig:flowchart}, we depict the flowchart of the above symbolic approach for determining whether an equilibrium is an $m$-snapback repeller at a given parameter point. The output of the approach is simply Yes or No.

\begin{figure}\centering
\begin{tikzpicture}[node distance=10pt]
  \node[draw, rounded corners, text width =10cm, text centered]                        (start)   {\textbf{Input:} an equilibrium $\xvar^*$ and a parameter point $\uvar$.};
  \node[draw, below=of start, text width =10cm, text centered]                         (step 1)  {Construct the following semi-algebraic system $S$, where $j$ equals to 1 or 2, and $T_k$ is defined as in \eqref{eq:T_k}.
  \begin{equation}\nonumber
\left\{\begin{split}
&{F}^{m}({\xvar}_0)-{{\xvar^*}}=0,~{\xvar}_0\neq {{\xvar^*}},~|{F}'({\xvar}_k)|\neq0,~ k=1,\ldots,m,\\
&(-1)^{i+j-1}T_{n-i}(1)|{\xvar^*}>0,~T_{n-i}(0)|{\xvar^*}\neq0,~i=0,\ldots,n,\\
&(-1)^{i+j-1}T_{n-i}(1)|\xvar_0>0,~T_{n-i}(0)|\xvar_0>0,~i=0,\ldots,n.
\end{split}\right.
\end{equation}};
  \node[draw, below=of step 1, text width =10cm, text centered]                        (step 2)  {Compute the border polynomial $BP$ of system $S$.};
  \node[draw, below=of step 2, text width =10cm, text centered]                        (step 3)  {Select at least one sample point from each region of the parameter set divided by $BP=0$.};
  \node[draw, below=of step 3, text width =10cm, text centered]                        (step 4)  {Determine the number of real solutions of $S$ at each selected sample point.};
  \node[draw, below=of step 4, text width =10cm, text centered]                        (step 5)  {Find one sample point $\pvar$ such that $\uvar$ lies in the same region as $\pvar$. This can be achieved if $BP(\uvar)\neq 0$.};
  \node[draw, diamond, aspect=2, below=of step 5, text width =5cm, text centered]  (choice)  
  {Does $S$ possess at least one real solution at $\pvar$?};
  \node[draw, rounded corners, right=20pt of choice, text width =3cm, text centered]       (step no)  
  {\textbf{Output:} No};
  \node[draw, diamond, aspect=2, below=20pt of choice, text width =5cm, text centered]       (step yes)  
  {Is there a repelling neighborhood $B_r(E)$ such that $\xvar_0\in B_r(E)$ at $\uvar$?};
   \node[draw, rounded corners, right=20pt of step yes, text width =3cm, text centered]       (step noBr)  
  {\textbf{Output:} No};
     \node[draw, rounded corners, below=20pt of step yes, text width =3cm, text centered]   (step yesBr)  
  {\textbf{Output:} Yes};
  
  \draw[->] (start)  -- (step 1);
  \draw[->] (step 1) -- (step 2);
  \draw[->] (step 2) -- (step 3);
  \draw[->] (step 3) -- (step 4);
  \draw[->] (step 4) -- (step 5);
  \draw[->] (step 5) -- (choice);
  \draw[->] (choice) -- node[above] {No} (step no);
  \draw[->] (choice) -- node[left] {Yes} (step yes);
  \draw[->] (step yes) -- node[above] {No} (step noBr);
  \draw[->] (step yes) -- node[left] {Yes} (step yesBr);
\end{tikzpicture}
\caption{Flowchart of the symbolic approach for determining whether a given equilibrium $\xvar^*$ is an $m$-snapback repeller at $\uvar$.}
\label{fig:flowchart}
\end{figure}
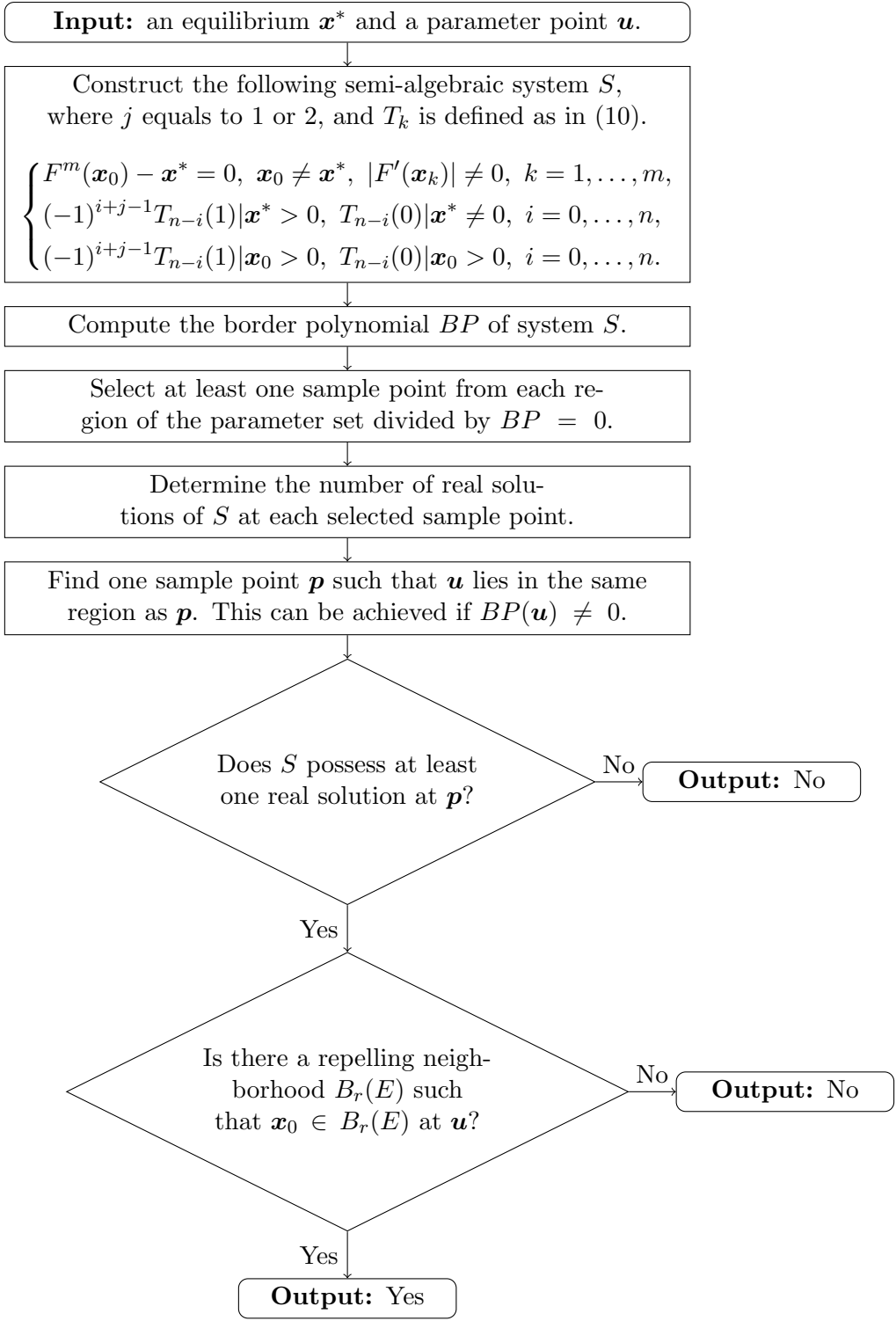

\subsection{Phase portraits}

In Figure \ref{fig:phase} (a) and (b), we depict the phase portraits for $(\rho,\mu)=(175/256,2)$ and $(\rho,\mu)=(5/32,59)$, respectively. The number of iterations is set as 2000, and the initial state is chosen to be $(x(0),y(0))=(0.4,0.01)$. From Figure \ref{fig:phase} (a), it is observed that the trajectory for $(\rho,\mu)=(175/256,2)$ converges to $(0.5,0.5)$, i.e., the positive equilibrium $E_1=(1-1/\mu,1-1/\mu)$. In Figure \ref{fig:phase} (b), we plot the following trajectory for $(\rho,\mu)=(5/32,59)$. 
\begin{align*}
&(0.4,0.01)\rightarrow 
(0.428765625,2.22093750)\rightarrow 
(-24.6360301,4.13182448)\rightarrow 
\\&(-1.40078649\cdot 10^2,-5.81879978\cdot 10^3)\rightarrow 
(-3.12186170\cdot 10^8,-1.87091533\cdot 10^5)\rightarrow 
\\&(-3.22951267\cdot 10^{11},-8.98461267\cdot 10^{17})\rightarrow
(-7.44167598\cdot 10^{36},-9.61493528\cdot 10^{23})\rightarrow 
\\&(-8.52245601\cdot 10^{48},-5.10520928\cdot 10^{74})\rightarrow 
(-2.40269773\cdot 10^{150},-6.69578614\cdot 10^{98})\rightarrow 
\\&(-4.13309308e\cdot 10^{198},-5.32194417\cdot 10^{301})\rightarrow\cdots	
\end{align*}
This trajectory seems to diverge. However, in Theorem \ref{thm:snapback}, we have proved the existence of snapback repellers, which implies the occurrence of chaos. Hence, in some cases, chaotic dynamics may be hard to observe through numerical simulations due to the limitations of the storage and expression of float-point numbers in existing systems of numerical computations. 

\begin{figure}[htbp]
  \centering
  \subfloat[$(\rho,\mu)=(175/256,2)$]{\includegraphics[width=0.45\textwidth]{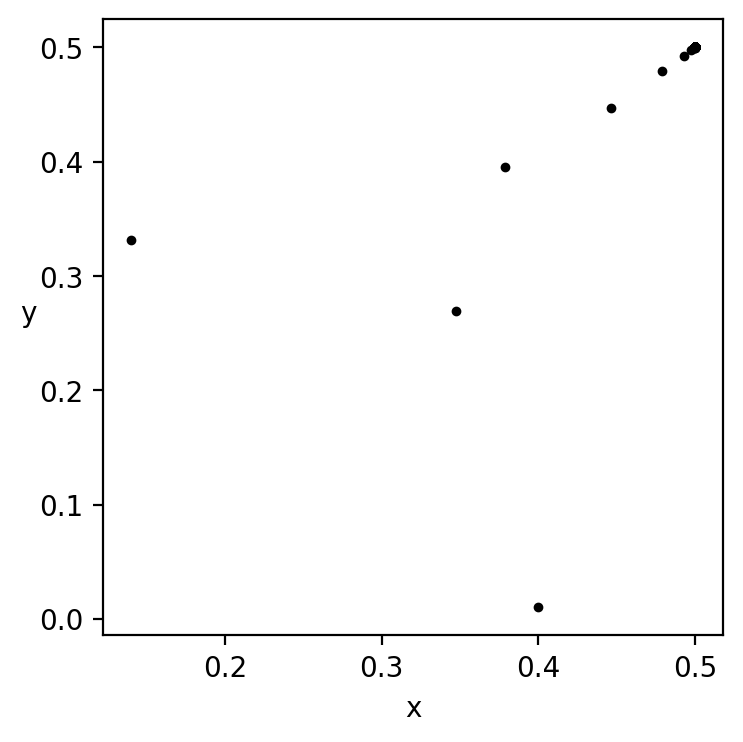}} 
  \subfloat[$(\rho,\mu)=(5/32,59)$]{\includegraphics[width=0.45\textwidth]{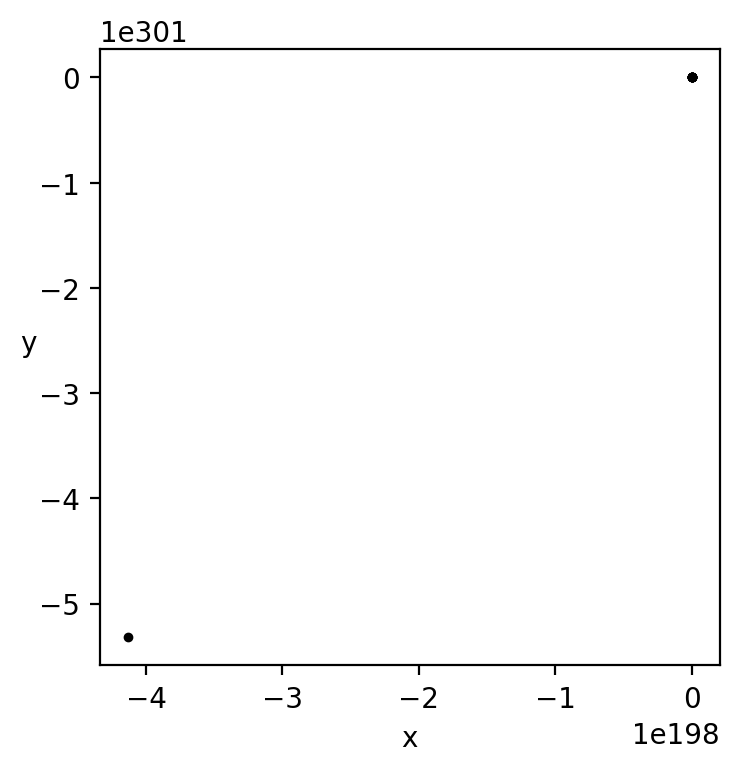}} \\
  \caption{Phase portraits with 2000 iterations, where the initial states are chosen to be $(x,y)=(0.4,0.01)$.}
\label{fig:phase}
\end{figure}



In Figure \ref{fig:complex-dyn} (a)--(c) and (g)--(i), we depict the phase portraits of six typical strange attractors discovered through numerical simulations, the first three of which are relatively simple, while the last three are much more complicated. According to our calculations, $E_0$ is not a 2-snapback repeller for all the fix cases. In Figure \ref{fig:complex-dyn} (d)--(f), we plot the positions of $\xvar_0$ corresponding to $E_0$ for the first three strange attractors (a)--(c), respectively. The computational results show that, for the first three cases, there are no non-zero 2-snapback repellers either. 

%

For the last three phase portraits (g)--(i), we can find non-zero equilibria ($E_1=(1-1/\mu,1-1/\mu)$) with their corresponding $\xvar_0$ in the same connected region defined by the repelling conditions, i.e., $C_i>0$, $i=2,1,0$. Figure \ref{fig:complex-dyn} (j)--(l) depict the positions of these non-zero equilibria (red points) and the corresponding $\xvar_0$ (blue points). It should be noted that in (j)--(l) one can not find a closed ball $B_r(E_1)$ that is covered by the yellow regions and contains any of the blue points. Actually, it is pointed out by Gardini et al.\ \cite{gardini_critical_2011} that the closed ball $B_r(\xvar^*)$ centered at the equilibrium $\xvar^*$ is not necessary and may be replaced by a repelling neighborhood $U$ containing $\xvar^*$. Interested readers can see \cite[Definition 2]{gardini_critical_2011} for more information. Accordingly, for Figure \ref{fig:complex-dyn} (j)--(l), one can prove the existence of snapback repellers according to \cite[Definition 2]{gardini_critical_2011}. This means chaotic dynamics indeed exist, which are confirmed by the last three phase portraits (g)--(i).


%

\begin{figure}[htbp]
  \centering

  \subfloat[$(\rho,\mu)=(63/128, 61/16)$]{\includegraphics[height=0.3\textwidth]{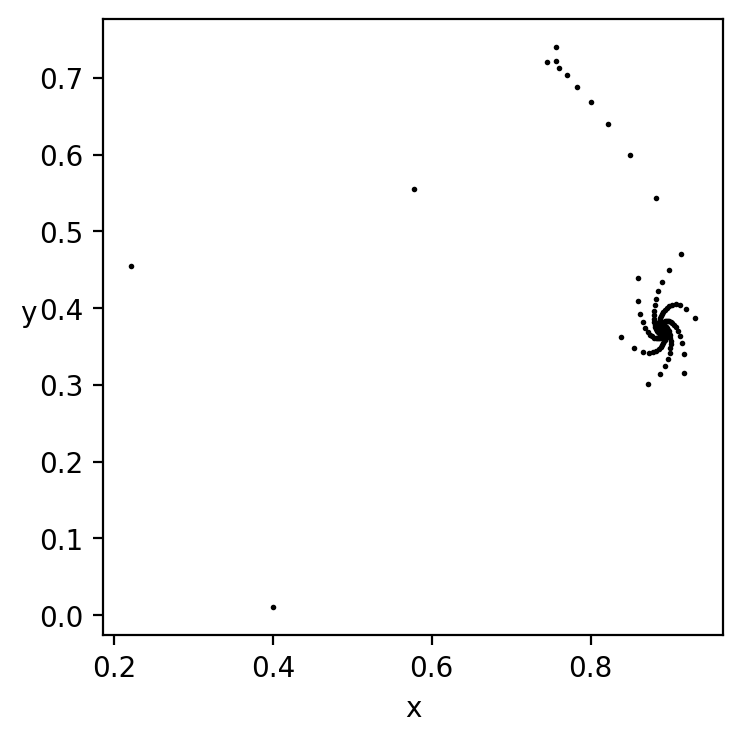}} 
  \subfloat[$(\rho,\mu)=(3/8, 5)$]{\includegraphics[height=0.3\textwidth]{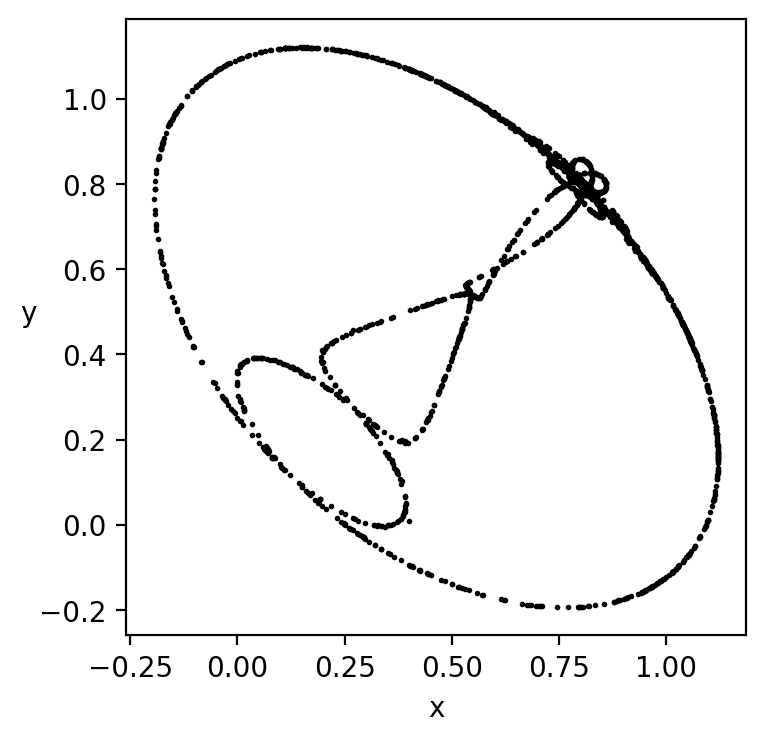}}
  \subfloat[$(\rho,\mu)=(69/128, b = 65/16)$]{\includegraphics[height=0.3\textwidth]{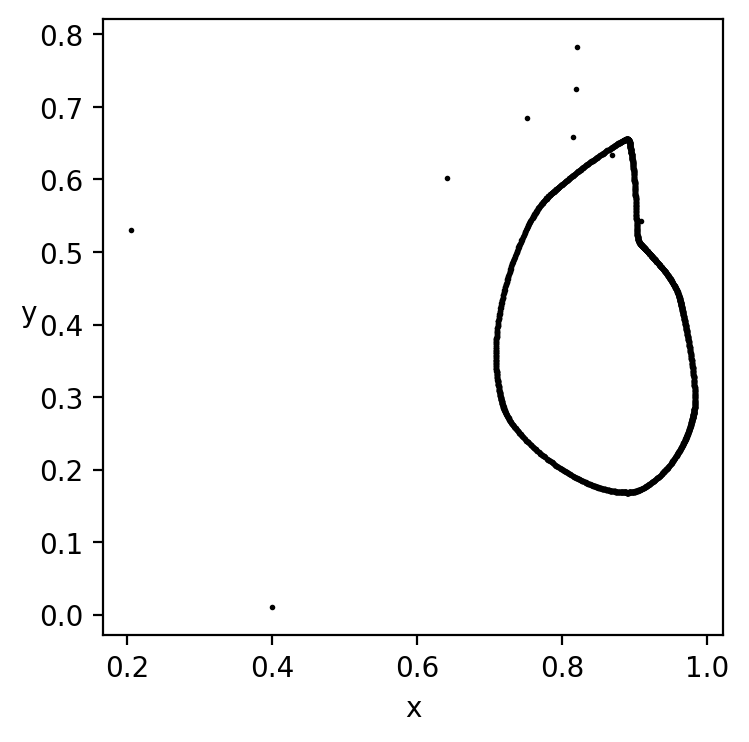}} \\ 

  \subfloat[$(\rho,\mu)=(63/128, 61/16)$]{\includegraphics[height=0.3\textwidth]{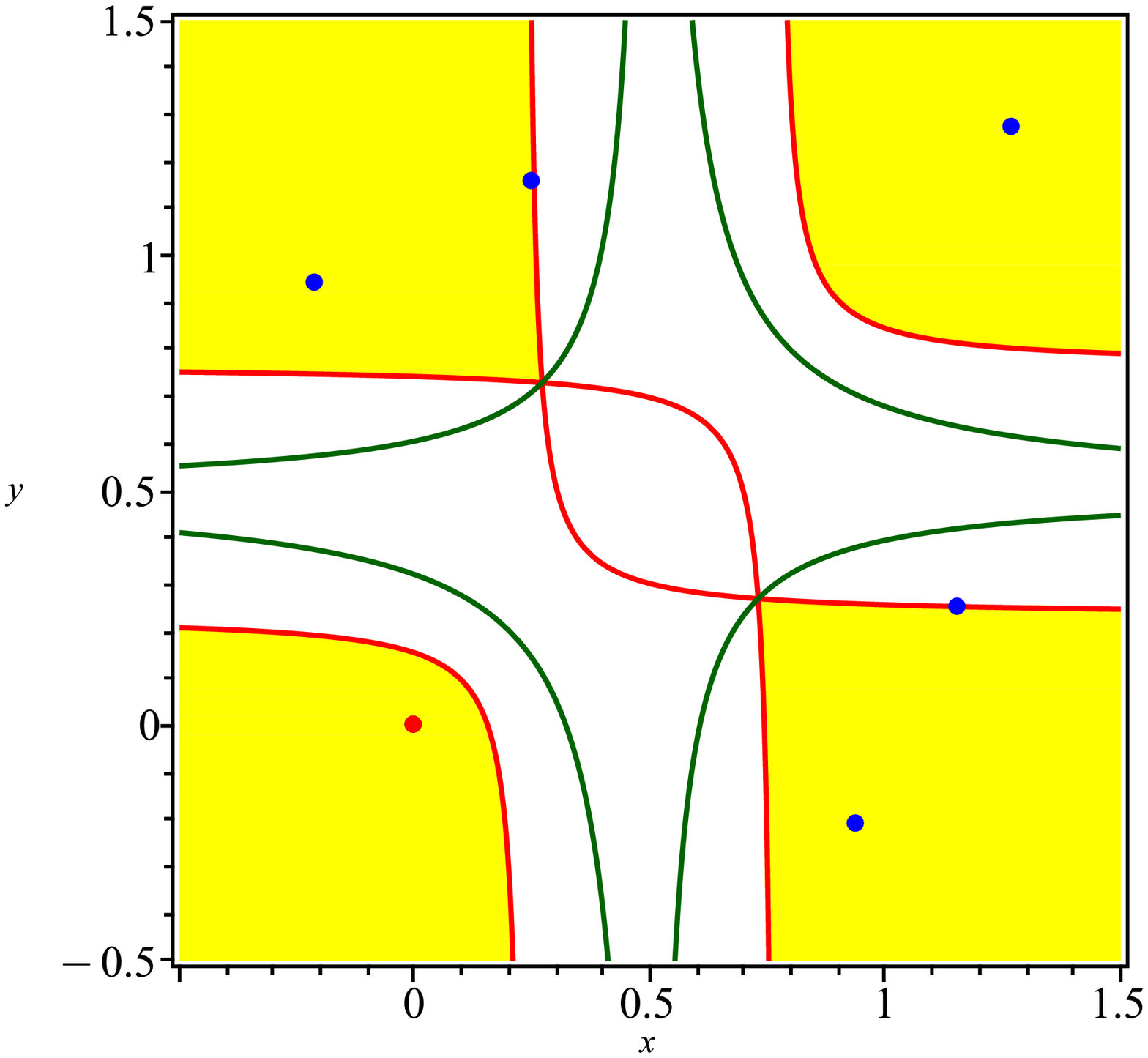}} 
  \subfloat[$(\rho,\mu)=(3/8, 5)$]{\includegraphics[height=0.3\textwidth]{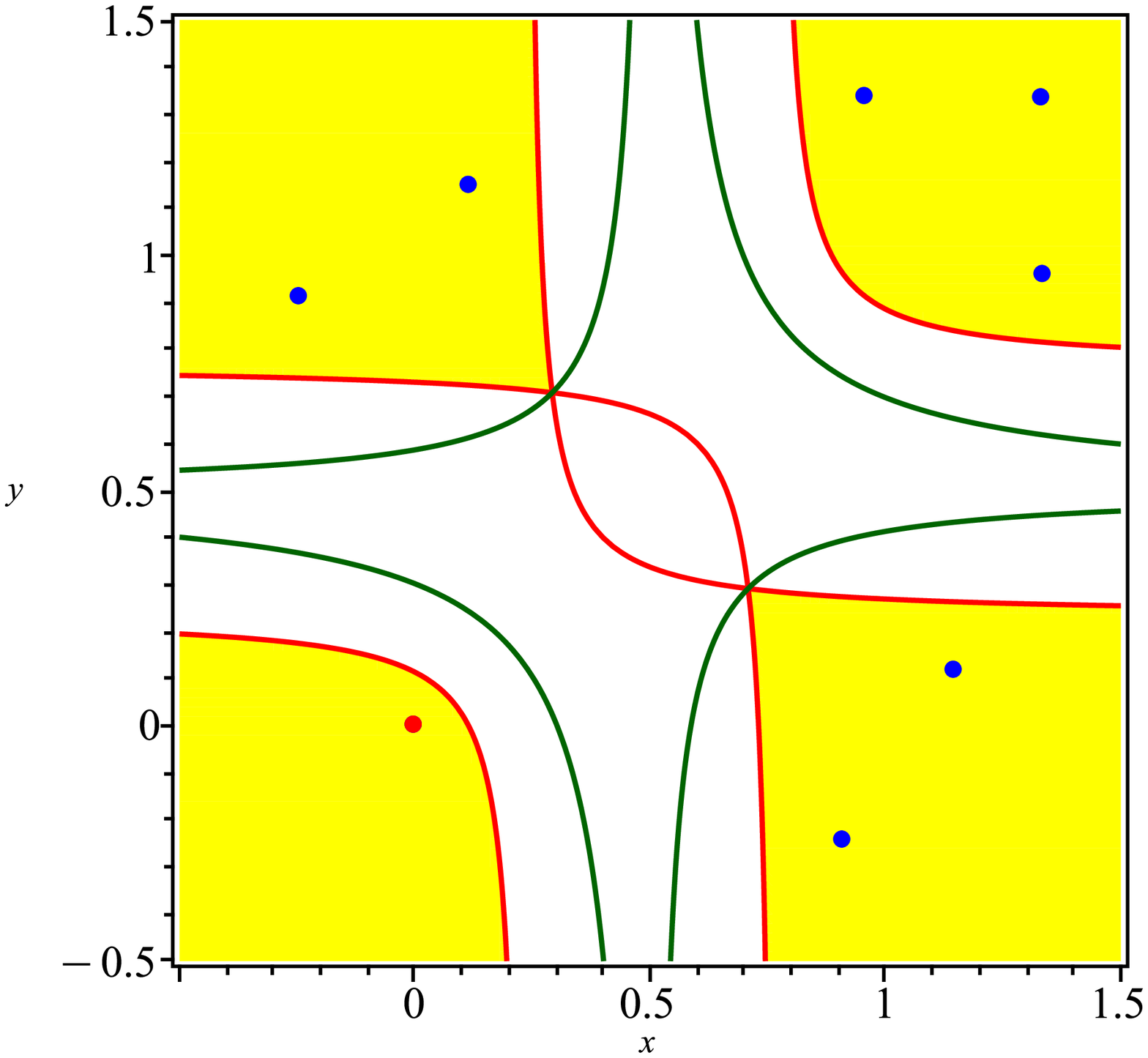}} 
  \subfloat[$(\rho,\mu)=(69/128, b = 65/16)$]{\includegraphics[width=0.3\textwidth]{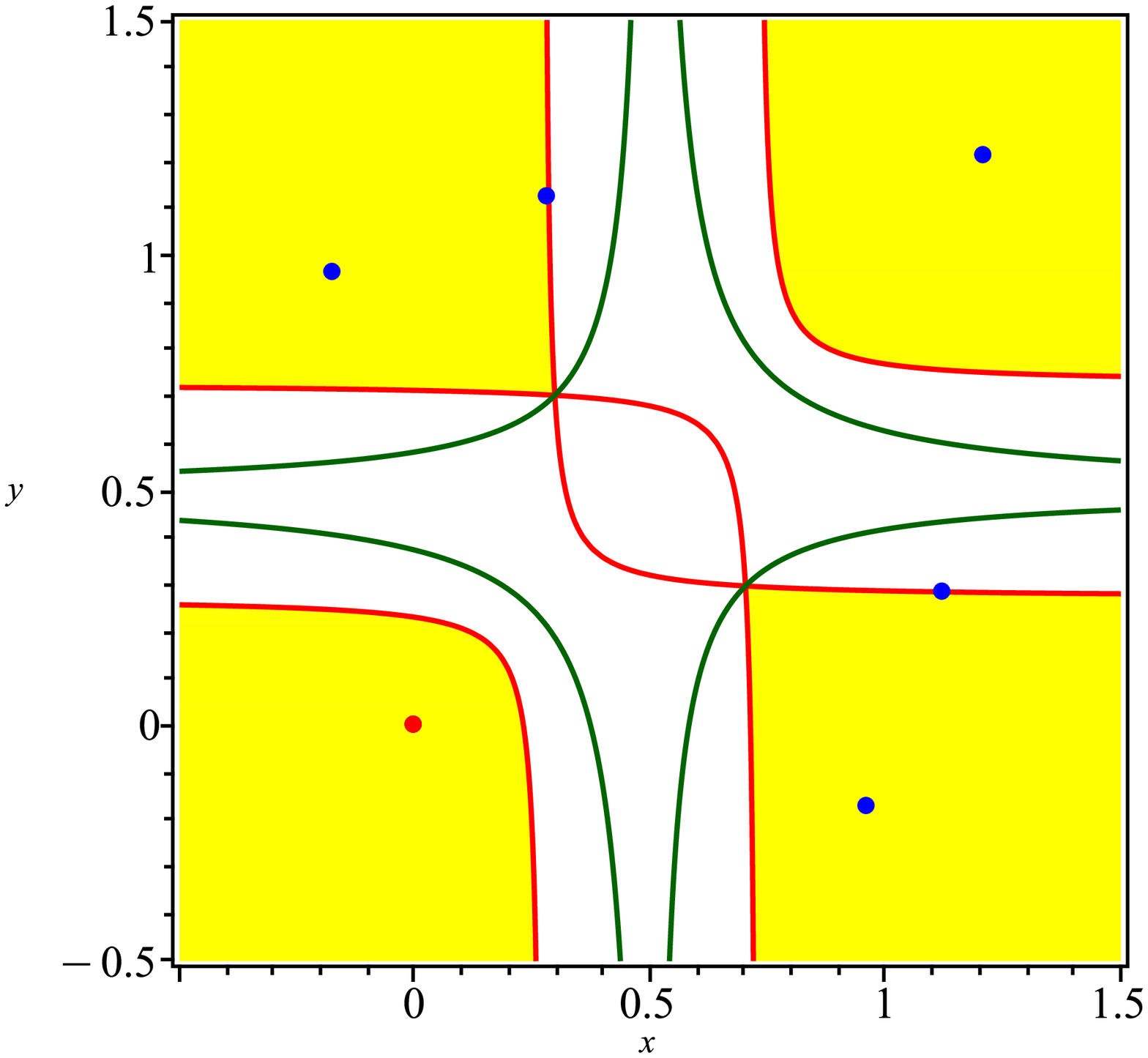}} \\

  \subfloat[$(\rho,\mu)=(15/16, 4)$]{\includegraphics[height=0.3\textwidth]{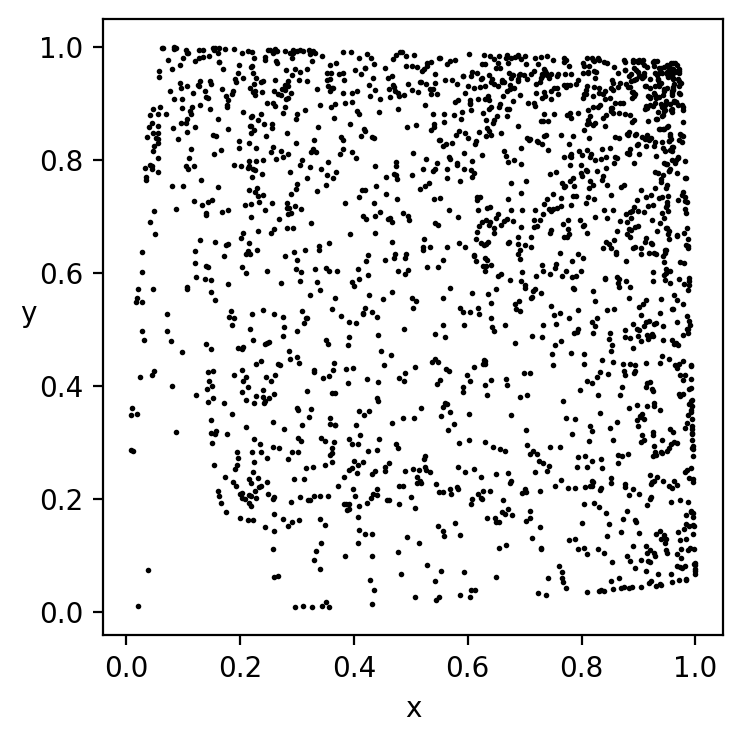}} 
  \subfloat[$(\rho,\mu)=(5/8, 17/4)$]{\includegraphics[height=0.3\textwidth]{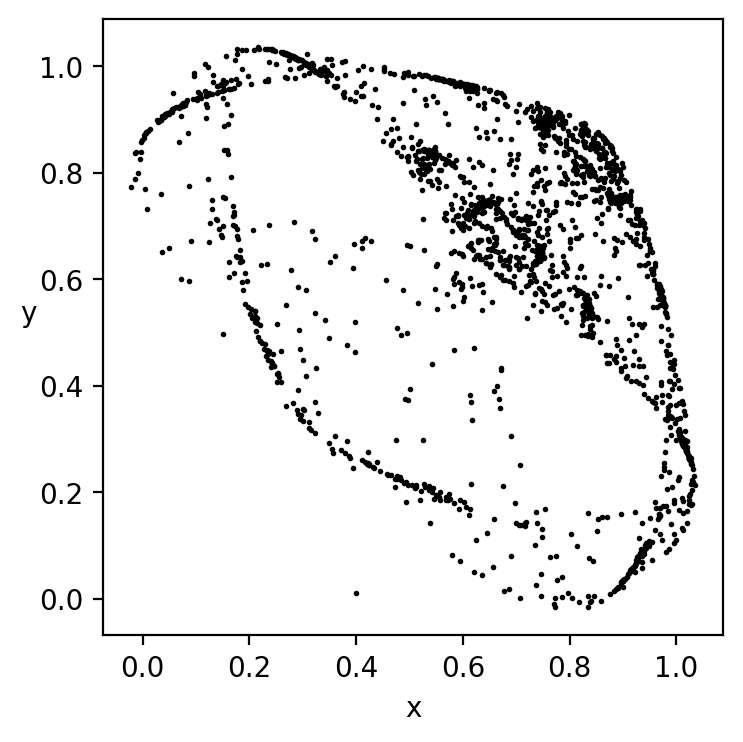}}
  \subfloat[$(\rho,\mu)=(389/512, 4413/1024)$]{\includegraphics[height=0.3\textwidth]{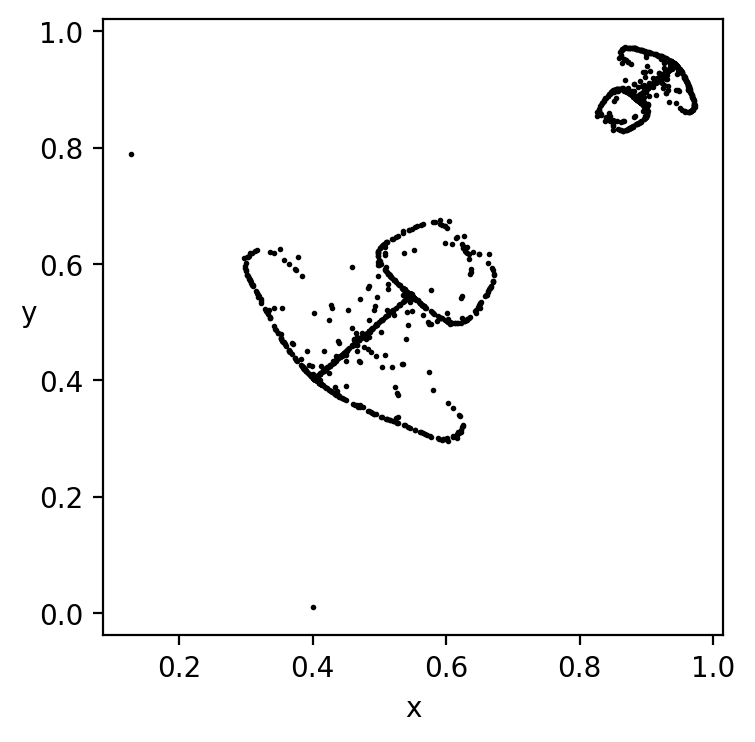}} \\

  \subfloat[$(\rho,\mu)=(15/16, 4)$]{\includegraphics[height=0.3\textwidth]{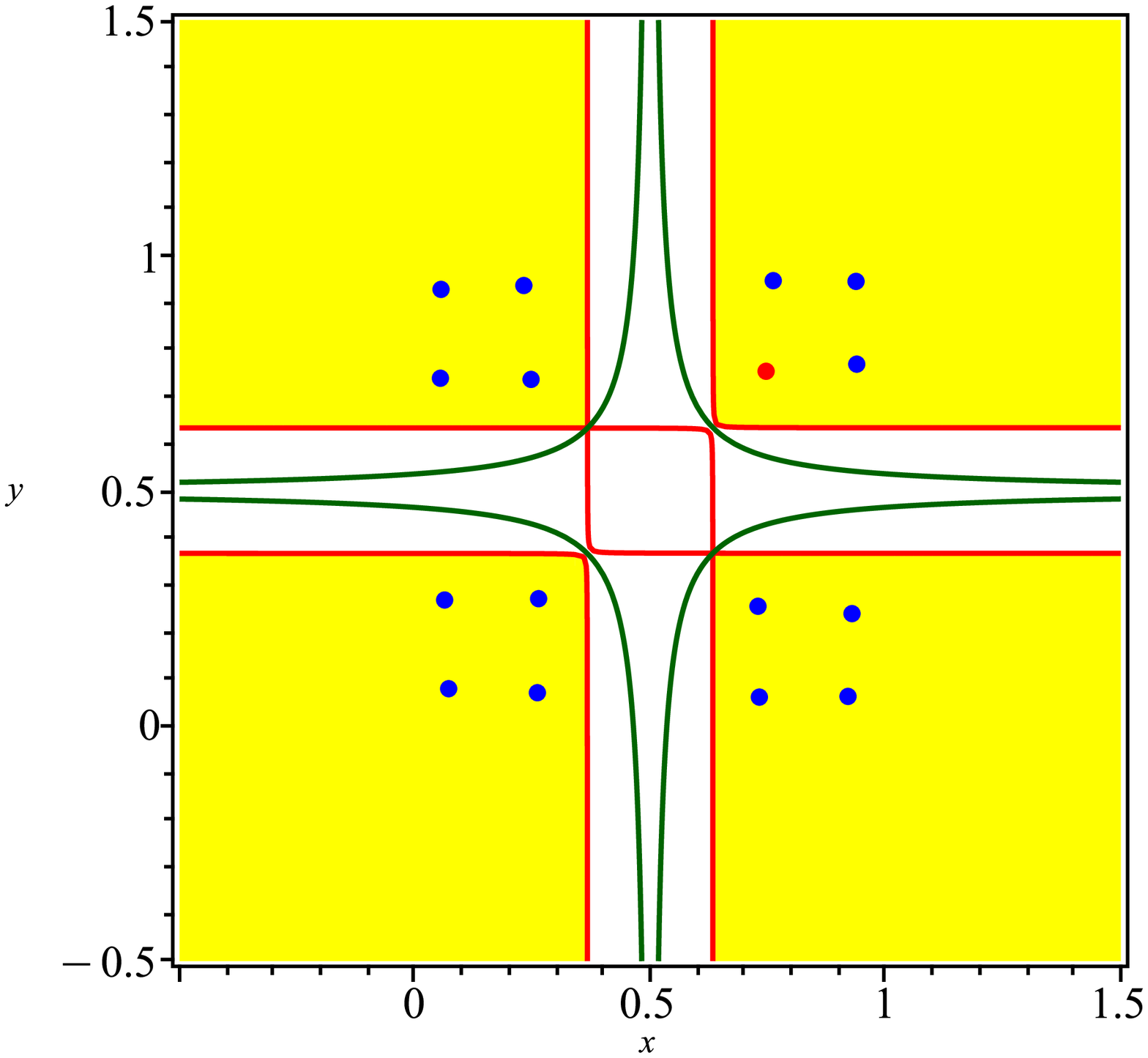}} 
  \subfloat[$(\rho,\mu)=(5/8, 17/4)$]{\includegraphics[width=0.3\textwidth]{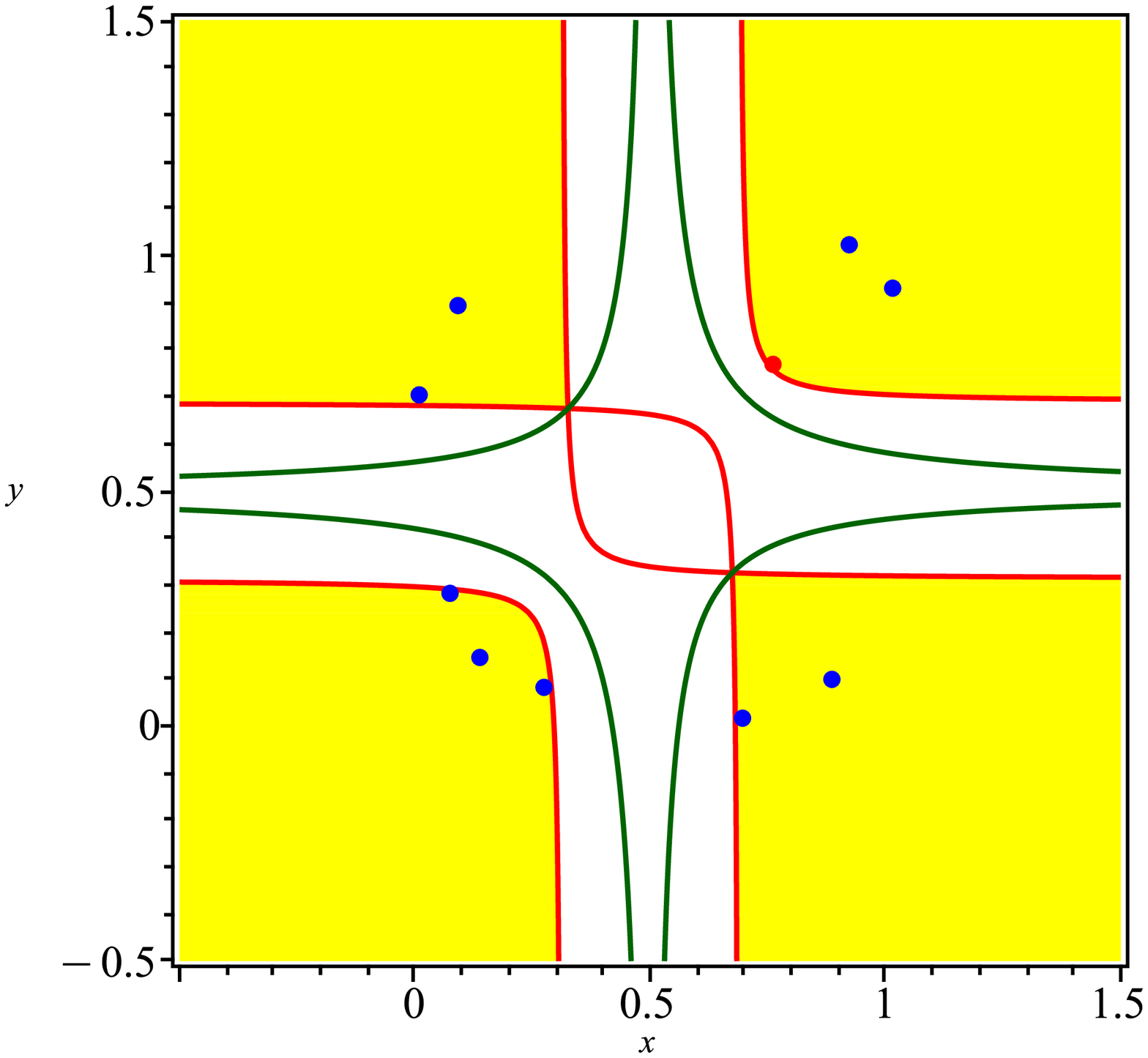}} 
  \subfloat[$(\rho,\mu)=(389/512, 4413/1024)$]{\includegraphics[height=0.3\textwidth]{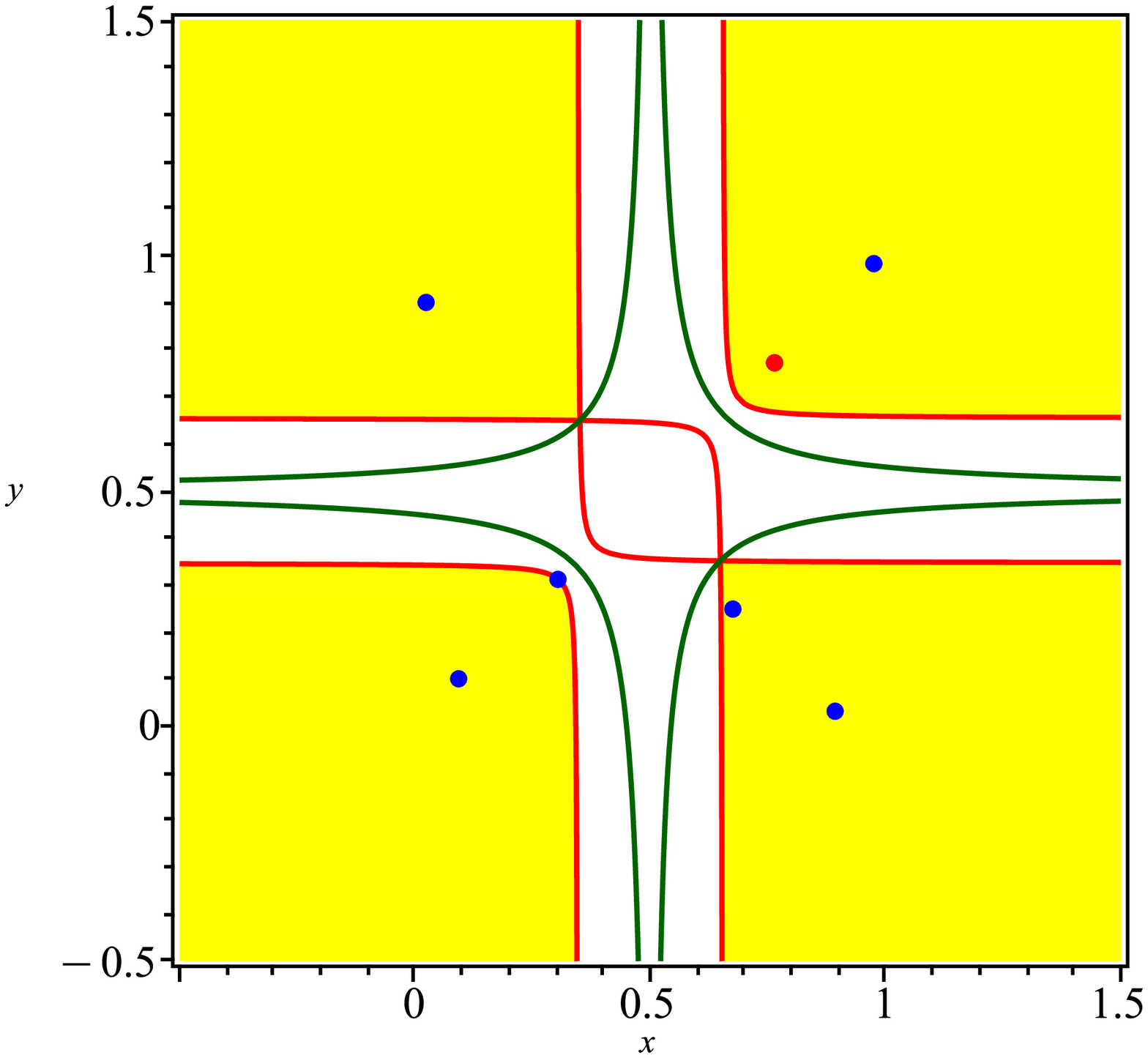}} 

  \caption{Phase portraits of strange attractors and positions of the equilibrium and the corresponding $x_0$. The considered fixed points are marked in red, while the corresponding points of $\xvar_0$ are marked in blue. The region of repellers is colored in yellow.}
\label{fig:complex-dyn}
\end{figure}

\section{Concluding Remarks}

This study filled the gap that there are almost no analytical investigations on the equilibria and their stability in the asymmetric duopoly model of Kopel. We derived the possible positions of the equilibria in Kopel's model. Specifically, all the equilibria should lie in $[0,1]\times [0,1]$ provided that $\mu_1\mu_2\geq 1$ (see Proposition \ref{prop:all-equi}). We explored the possibility of the existence of multiple positive equilibria and established necessary and sufficient conditions for a given number of equilibria to exist (see Theorem \ref{thm:equi-num}). Furthermore, if the two duopolists adopt the best response reactions ($\rho_1=\rho_2=1$) or homogeneous adaptive expectations ($\rho_1=\rho_2$), we established rigorous conditions for distinct numbers of positive equilibria to exist for the first time (see Theorems \ref{thm:stable-equiv} and \ref{thm:ab-one}). In the general case of $\rho_1\neq \rho_2$, however, we failed to obtain the complete conditions that a given number of stable equilibria exist. We explained the essential difficulty is that the algebraic description of a region bounded by algebraic varieties is expensive to compute.

In the current literature, the existence of chaos in Kopel's model seems to be supported by observations through numerical simulations, which, however, is challenging to prove. In our study, we rigorously proved the existence of snapback repellers in Kopel's map, which implies the existence of chaos in the sense of Li--Yorke according to Marotto’s theorem (see Theorem \ref{thm:snapback}). Furthermore, we noticed that the algorithm of \cite{huang_analysis_2019} for finding snapback repellers is insufficient, and we fixed this flaw by introducing more steps of further validation. The flowchart of the entire approach for determining whether a given equilibrium is a snapback repeller at a parameter point is reported in Figure \ref{fig:flowchart}.

Our investigation employed several tools based on symbolic computations such as the triangular decomposition and resultant methods. The results produced by symbolic computations are exact and rigorous, and thus can be used to discover and prove mathematical theorems analytically.

\section*{Acknowledgements}


This work has been supported by Philosophy and Social Science Foundation of Guangdong (No.\ GD21CLJ01), National Natural Science Foundation of China (No.\ 12101032 and No.\ 12131004), Beijing Natural Science Foundation (No.\ 1212005), Social Development Science and Technology Project of Dongguan (No.\ 20211800900692).

\section*{Data Availability}

The authors declare that the manuscript has no associated data.

\section*{Conflict of Interest}

The authors declare that they have no conflict of interest.

\bibliographystyle{abbrv}

%
%


%
%
%
%
%
%
%
%
%
%
%
%
%
%
%
%

\end{CJK}
\end{document}